\documentclass{amsart}

\usepackage{amssymb}
\usepackage{latexsym}
\usepackage{amsmath}

\def\wt{\widetilde}

\def\ov{\overline}
\def \im{{\rm Im}}
 \def\up{\upharpoonright}
\def\cH{\mathcal H}

\def\cD{\mathcal D}  
\def\cK{\mathcal K} \def\cL{\mathcal L}
\def\cM{\mathcal M} \def\cN{\mathcal N} \def\cP{\mathcal P} 
 \def\cT{\mathcal T} \def\cI{\mathcal I}

\def \gH{\mathfrak H}   \def \gN{\mathfrak N}

\def \bC{\mathbb C}    \def\bR{\mathbb R}
\def\bH{\mathbb H} \def\bK{\mathbb K}
\def \l{\lambda}
\def \a{\alpha} \def \b{\beta}    
 \def \t{\theta} \def\g {\gamma}
\def\d {\delta}  
\def \f{\varphi}  \def \G{\Gamma} \def\D {\Delta}

\def \C{\widetilde {\mathcal C}}
\def \CA{\C(\cH_0,\cH_1)}
\def \CB{\C(\cH_1,\cH_0)}

\def \cd {\cdot}

\def\BR {\G:\gH^2\to \cH_0\oplus\cH_1}
\def\lb {\left\{} \def\rb {\right\}}
\def\HH {\cH_0\oplus\cH_1}
\def\AC {AC(\cI)} \def\ACf {AC_0(\cI)}  \def\LI {L_\Delta^2(\cI)}
\def\lI {\cL_\Delta^2(\cI)}
\def\tma{\cT_{max}} \def\tmi{\cT_{min}} \def\Tma{T_{max}} \def\Tmi{T_{min}}

\def \dom {{\rm dom}\,}  \def \ran {{\rm ran}\,}  \def \ker{{\rm ker\,}}
 \def \mul {{\rm mul}\,} \def \sign {{\rm sign}\,}

\def \exa { {Ext}_A}

\def \CR {\bC\setminus\bR}

\def\bt{\{\cH,\G_0,\G_1\}}
\def\bta{\{\cH_0\oplus \cH_1,\Gamma _0,\Gamma _1\}}

\newtheorem{theorem}{Theorem}[section]
\newtheorem{proposition}[theorem]{Proposition}
\newtheorem{corollary}[theorem]{Corollary}
\newtheorem{lemma}[theorem]{Lemma}

\theoremstyle{definition}
\newtheorem{example}[theorem]{Example}
\theoremstyle{definition}
\newtheorem {definition} [theorem]{Definition}
\theoremstyle{remark}
\newtheorem{remark}[theorem]{Remark}
\numberwithin{equation}{section}

\begin{document}
\title [Boundary relations and boundary conditions ]
{Boundary relations and boundary conditions for general (not necessarily
definite) canonical systems with possibly unequal deficiency indices}
\author {Vadim Mogilevskii}
\address{Department of Calculus\\
Lugans'k National   University  \\
2 Oboronna, Lugans'k, 91011\\
Ukraine}

\begin{abstract}
 We investigate in the paper general (not necessarily definite) canonical
systems of differential equation in the framework of extension theory of
symmetric linear relations. For this aim we first introduce  the new notion of
a boundary relation $\G:\gH^2\to\HH$ for $A^*$, where $\gH$ is a Hilbert space,
$A$ is a symmetric linear relation  in $\gH, \cH_0$ is a boundary Hilbert space
and $\cH_1$ is a subspace in $\cH_0$. Unlike known concept of a boundary
relation (boundary triplet) for $A^*$ our definition of $\G$ is applicable to
relations $A$ with possibly unequal deficiency indices $n_\pm(A)$. Next we
develop the known results on minimal and maximal relations induced by the
general  canonical system $ J y'(t)-B(t)y(t)=\D (t)f(t)$ on an interval
$\cI=(a,b),\; -\infty\leq a<b\leq\infty $ and then by using a special (so
called decomposing) boundary relation for $\Tma$ we describe in terms of
boundary conditions proper extensions of $\Tmi$ in the case of the regular
endpoint $a$ and arbitrary (possibly unequal) deficiency indices $n_\pm
(\Tmi)$. If the system is definite, then decomposing boundary relation $\G$
turns into the decomposing boundary triplet $\Pi=\bt$ for $\Tma$. Using such a
triplet we show that self-adjoint decomposing boundary conditions exist only
for Hamiltonian systems; moreover, we describe all such conditions in the
compact form. These results are generalizations of the known results by
Rofe-Beketov on regular differential operators.  We characterize also all
maximal dissipative and accumulative separated boundary conditions, which exist
for arbitrary (not necessarily Hamiltonian) definite canonical systems.
\end{abstract}

\email{vim@mail.dsip.net}

\subjclass[2000]{}


\maketitle
\section{Introduction}
Assume that $\gH$ is a Hilbert space, $A$ is a closed symmetric linear relation
in $\gH$ and $A^*$ is the adjoint linear relation of $A$. Moreover, denote by
$[\gH_1, \gH_2]$ the set of all bounded operators between $\gH_1$ and $\gH_2$
and let $[\gH]=[\gH, \gH]$.

Recall \cite{GorGor,Mal92} that a triplet $\Pi=\bt $, where $\cH$ is an
auxiliary Hilbert space and $\G_0,\G_1: A^*\to \cH$ are (boundary) linear maps,
is called a boundary triplet for $A^*$ if the map
$\G:=(\G_0\;\;\G_1)^\top:A^*\to \cH\oplus\cH$ is surjective and the following
''abstract Green's identity'' holds
\begin {equation}\label{1.1}
(f',g)-(f,g')=(\G_1  \hat f, \G_0 \hat g)-(\G_0\hat f,\G_1\hat g), \quad \hat
f=\{f,f'\},\;\hat g=\{g,g'\}\in  A^*.
\end{equation}

In \cite{DM91,Mal92} an  abstract Weyl function $M_\Pi(\l)$ was associated with
a boundary triplet $\Pi$. This function is defined for all $\l\in\CR$ by the
equality
\begin{gather}\label{1.2}
\G_1\{f_\l,\l f_\l\}=M_\Pi(\l)\G_0\{f_\l,\l f_\l\}, \quad f_\l\in\ker (A^*-\l).
\end{gather}
It turns out that $M(\l)$ is a Nevanlinna $[\cH]$-valued function, i.e.,
$M(\l)$ is holomorphic on $\CR$, $M^*(\l)=M(\ov\l)$ and $\im\l\cd \im M(\l)\geq
0, \; \l\in\CR $. Moreover, the Nevanlinna function $M(\l)$ is uniformly
strict, that is $0\in\rho (\im M(\l)), \; \l\in\CR$.

By choosing a suitable boundary triplet for a concrete problem one can
parametrize various classes of extensions $\wt A\supset A$ in the most
convenient form. Moreover, the Weyl function enables to characterize spectra of
extensions $\wt A$ in the similar way as classical $m$-functions in the
spectral theory of Sturm-Liouville operators and Jacobi matrices. These and
other reasons made a boundary triplet and the corresponding Weyl function the
convenient tools in the extension theory of symmetric operators (linear
relations) and its applications (see \cite{GorGor,DM91,Mal92} and references
therein). At the same time the theory of boundary triplets and their Weyl
functions was developed in \cite{GorGor,DM91,Mal92} only for symmetric
relations $A$ with equal deficiency indices $n_+(A)=n_-(A)$.

To cover the case $n_+(A)\neq n_-(A)$ we generalized in \cite{Mog06.2}
definition of a boundary triplet as follows. Assume that $\cH_0$ is a Hilbert
space, $\cH_1$ is a subspace in $\cH_0$ and $\G_j: A^*\to \cH_j, \;
j\in\{0,1\}$ are linear maps. Then a collection $\Pi=\bta$ is a boundary
triplet (a $D$-triplet in terminology of \cite{Mog06.2}) for $A^*$ if the map
$\G:=(\G_0\;\;\G_1)^\top:A^*\to \cH_0\oplus\cH_1$ is surjective and the
identity
\begin {equation}\label{1.3}
(f',g)-(f,g')=(\G_1  \hat f, \G_0 \hat g)-(\G_0\hat f,\G_1\hat g)+i(P_2\G_0\hat
f,P_2\G_0\hat g), \quad \hat f=\{f,f'\},\;\hat g=\{g,g'\}\in  A^*
\end{equation}
holds in place of \eqref{1.1} (here $P_2$ is the orthoprojector in $\cH_0$ onto
 $\cH_2:=\cH_0\ominus\cH_1$). Associated with such a triplet $\Pi$ is the Weyl
 function $M_{\Pi +} (\l)$ defined for all $\l\in\bC_+$ by
\begin{gather}\label{1.4}
\G_1\{f_\l,\l f_\l\}=M_{\Pi +}(\l)\G_0\{f_\l,\l f_\l\}, \quad f_\l\in\ker
(A^*-\l)
\end{gather}
The function $M_{\Pi +}(\l)$ is holomorphic on $\bC_+$, takes on values in
$[\cH_0,\cH_1]$ and possesses a number of properties similar to those of the
Weyl function \eqref{1.2}. In particular, the function $M_\Pi(\l)=M_{\Pi
+}(\l)\up \cH_1$ is a uniformly strict Nevanlinna function with values in
$[\cH_1]$.

A boundary triplet $\bta$ for $A^*$ enables to parametrize efficiently  all
proper extensions of $A$. Namely, if $\cK$ is a Hilbert space and
$\{(C_0,C_1);\cK\}$ is a pair of operators $C_j\in [\cH_j,\cK]$, then the
equality (the abstract boundary condition)
\begin {equation}\label{1.5}
 \wt A= \{ \hat f\in A^*: C_0\G_0  \hat f + C_1 \G_1 \hat f =0\}
 \end{equation}
defines the proper extension $A\subset \wt A\subset A^*$ and conversely  each
such an extension $\wt A$ admits a unique representation  \eqref{1.5}.
Moreover, the extension $\wt A$ is maximal dissipative, maximal accumulative or
self-adjoint if and only if the operator pair  $\{(C_0,C_1);\cK\}$ belongs to
one of the special classes introduced in \cite{Mog06.1}.

It turns out that each boundary triplet $\Pi=\bta$ satisfies the relation
\begin{gather}\label{1.5a}
\dim\cH_1=n_-(A)\leq n_+(A)=\dim\cH_0
\end{gather}
and, therefore, it is applicable  to symmetric relations $A$ with unequal
deficiency indices. Clearly, in the case $\cH_0=\cH_1=:\cH$ such a triplet $\Pi
$ and the corresponding Weyl function $M_\Pi(\l)$ turn into the similar objects
in the sense of \cite{GorGor,Mal92}.

In \cite{DM95} the notion of a boundary triplet $\Pi=\bt$ for $A^*$ has been
extended to the case where the corresponding Weyl function $M_\Pi(\l)$ is a
(not necessarily uniformly strict) Nevanlinna function such that
$0\notin\sigma_p(\im M(i))$.  Next,  the concepts of a boundary relation and
its Weyl family which generalize the above notions of a boundary triplet and
its Weyl function were introduced in \cite{DM06}. According to \cite{DM06} a
boundary relation for $A^*$ is a (possibly multivalued) linear map
$\G:=(\G_0\;\;\G_1)^\top:\gH^2\to \cH\oplus\cH$ such that $\dom\G$ is dense in
$A^*$, the Green's identity \eqref{1.1} holds and a certain maximality
condition is satisfied. The Weyl function of the boundary relation $\G$ is
defined by
\begin{gather*}
M(\l)=\{\{\G_0\{f_\l,\l f_\l\}, \G_1\{f_\l,\l f_\l\}\}: f_\l\in\ker (A^*-\l)
\}, \quad\l\in\CR
\end{gather*}
and now it belongs to the class of Nevanlinna families; moreover, if the map
$\G_0$ is surjective, then $M(\l)$ is a Nevanlinna operator function. In the
paper \cite{DM09} the Weyl function was used for description of various classes
of the exit space extensions $\wt A(=\wt A^*)\supset A$.

In the present paper the new concept of a boundary relation for $A^*$ with
possibly unequal boundary spaces $\cH_0$ and $\cH_1$ is  introduced. Roughly
speaking this relation is a (possibly multivalued) linear map
$\G:=(\G_0\;\;\G_1)^\top:\gH^2\to \cH_0\oplus\cH_1$  such that
$\ov{\dom\G}=A^*$, the Green's identity \eqref{1.3} holds and a certain
maximality condition is satisfied (here as before $\cH_0$ is a Hilbert space
and $\cH_1$ is a subspace in $\cH_0$). Moreover, by means of the equality
\begin{gather}\label{1.6}
M_+(\l)=\{\{\G_0\{f_\l,\l f_\l\}, \G_1\{f_\l,\l f_\l\}\}: f_\l\in\ker (A^*-\l)
\}, \quad\l\in\bC_+
\end{gather}
we associate with a boundary relation $\G$   the  Weyl family $M_+(\l)$.

In the paper we study substantially the boundary relations $\BR$ with
$\dim\cH_0 <\infty$. We show that in this case $\dom\G=A^*$ and there is a
boundary triplet $\Pi_\G=\{\cK_0\oplus\cK_1,G_0, G_1\}$ for $A^*$ with
$\cK_j\in\cH_j, \; j\in \{0,1\}$ such that $\G$ can be represented roughly
speaking as a direct sum of (the graph of) the operator $G=(G_0\;\;G_1)^\top$
and $ \mul\G$.

The multivalued part $\mul \G$ which is a linear relation from $\cH_0$ to
$\cH_1$ is of importance in our considerations. If $\mul \G$ is the operator,
then the corresponding Weyl family $M_+(\l)$ is the holomorphic operator
function with values in $[\cH_0, \cH_1]$, which admits the block representation
by means of the Weyl function $M_{\Pi_\G +} (\l)$ of the boundary triplet
$\Pi_\G$ and $\mul \G $. In the case $\cH_0=\cH_1=:\cH$ one has also
$\cK_0=\cK_1=:\cK$ and the mentioned representation of $M(\l)$ is
\begin{gather}\label{1.7}
M(\l)=\begin{pmatrix} M_{\Pi_\G}(\l) & F \cr F^* & F'\end{pmatrix}:
\cK\oplus\cK^\perp\to \cK\oplus\cK^\perp,
\end{gather}
where $F$ and $F'$ are the operators defined  in terms of $\mul\G$. The
equality \eqref{1.7} shows that $M(\l)$ is a Nevanlinna function and
$M_{\Pi_\G}(\l)$ is the uniformly strict part of $M(\l)$.

Note that for the boundary relation $\BR$ with $\dim\cH_0 <\infty$ the
equalities
\begin{gather*}
\dim\cH_0=n_+(A)+\dim (\mul\G), \qquad \dim\cH_1=n_-(A)+\dim (\mul\G)
\end{gather*}
are valid (c.f. \eqref{1.5a}), so that $n_-(A)\leq n_+(A)<\infty$. At the same
time in the case of unequal deficiency indices $n_+(A)\neq n_-(A)$ each
boundary relation $\G:\gH^2\to \cH^2$ for $A^*$ in the sense of \cite{DM06}
satisfies the equality $\dim \cH=\infty$ (see \cite[Proposition 3.2]{DM09}).
This assertion shows that in the case  $n_+(A)\neq n_-(A)$ our definition of a
boundary relation is more natural and convenient for applications. Observe also
that other generalizations of boundary triplets can be found e.g. in
\cite{BehLan07}.

Next by using the concept of a boundary relation we investigate in the paper
linear relations induced by a general (not necessarily definite) canonical
system of differential equations with possibly unequal deficiency indices. Such
a system is of the form
\begin{gather}\label{1.8}
J y'(t)-B(t)y(t)=\D (t)f(t), \quad t\in\cI,
\end{gather}
where $J$ is an operator in the finite-dimensional Hilbert space $\bH$ such
that $J^*=J^{-1}=-J$ and $B(t)$ and $\D(t)$ are locally integrable
$[\bH]$-valued functions defined on an interval $\cI=(a,b),\; -\infty\leq a< b
\leq\infty,$ and such that $B(t)=B^*(t)$ and $ \D(t)\geq 0$ a.e. on $\cI$.
Without loss of generality we assume that
\begin{gather}\label{1.9}
\bH=H\oplus\hat H\oplus H
\end{gather}
with the Hilbert spaces $H$ and $\hat H$ and the operator $J$ is
\begin{gather} \label{1.10}
J=\begin{pmatrix} 0 & 0&-I_H \cr 0& i I_{\hat H}&0\cr I_H&
0&0\end{pmatrix}:H\oplus\hat H\oplus H \to H\oplus\hat H\oplus H.
\end{gather}
The canonical system \eqref{1.8} is called Hamiltonian if $\hat H=\{0\}$, in
which case the operator $J$ takes the form
\begin{gather*}
J=\begin{pmatrix} 0 & -I_H \cr  I_H& 0\end{pmatrix}:H\oplus H \to H\oplus H.
\end{gather*}
Clearly, the Hamiltonian system is a particular case of the system \eqref{1.8}.

Denote by $\lI$ the semi-Hilbert space of $\bH$-valued Borel  functions $f(t)$
on $\cI$ with $\int_\cI (\D(t)f(t),f(t))\, dt < \infty $ and let $(f,g)_\D$ be
the semi-definite inner product in $\lI$. Assume also that $\LI$ is the
corresponding Hilbert space of equivalence classes and $\pi$ is the quotient
map from $\lI$ onto $\LI$, so that the inner product in $\LI$ is $(\wt f, \wt g
)(=(\pi f,\pi g))=(f,g)_\D, \;\; \wt f,\wt g\in \LI$.

The null manifold $\cN$ of the system \eqref{1.8} plays an essential role in
our considerations. Recall \cite{KogRof75} that $\cN$ is defined as the set of
all solutions of the equation $J y'(t)-B(t)y(t)=0$ such that $\D(t)y(t)=0$ a.e.
on $\cI$. The system \eqref{1.8} is said to be definite if $\cN=\{0\}$ and
indefinite in the opposite case.

As is known the extension theory of symmetric relations is the natural
framework for boundary value problems involving canonical systems of
differential equations (see \cite{Orc,LT82,DLS88,DLS93,HSW00,BHSW10,LesMal03}
and references therein). This framework is based on the concept of minimal and
maximal relations which are defined as follows. Let $\tma$ be the set of all
pairs $\{y,f\}\in\lI\times \lI $ satisfying the system \eqref{1.8} and let
$\cT_0$ be the set of all $\{y,f\}\in\tma$ such that $y$ has compact support.
Then $\tma$ and $\cT_0$ are linear relations in $\lI$ and the Lagrange's
identity
\begin {equation*}
(f,z)_\D-(y,g)_\D=[y,z]_b -[y,z]_a,\quad \{y,f\}, \; \{z,g\} \in\tma.
\end{equation*}
holds with
\begin {equation}\label{1.11}
[y,z]_a:=\lim_{t \downarrow a}(J y(t),z(t)),\qquad [y,z]_b:=\lim_{t \uparrow
b}(J y(t),z(t)), \quad y,z \in\dom\tma.
\end{equation}
By using \eqref{1.11} introduce also the linear relation $\tmi$ in $\lI$ by
\begin {equation}\label{1.12}
\tmi=\{\{y,f\}\in\tma:\,[y,z]_a= [y,z]_b=0, \; z\in\dom\tma\}.
\end{equation}
Moreover, in the case of the regular endpoint $a$ (that is, if $a\neq -\infty$
and $B(t)$ and $\D(t)$ are integrable on $(a,\b), \;\b\in\cI$) let
\begin {equation} \label{1.13}
\cT_a=\{\{y,f\}\in\tma:\, y(a)=0 \;\;\text{and}\;\; [y,z]_b=0, \;
z\in\dom\tma\}.
\end{equation}
All the above relations in $\lI$ naturally generate by means of the equalities
\begin{gather}
\Tmi=\{\{\pi y, \pi f\}:\{y,f\}\in\tmi\}, \;\; \;\;\;T_a=\{\{\pi y, \pi
f\}:\{y,f\}\in\cT_a\},\label{1.14} \\
 T_0=\{\{\pi y, \pi f\}:\{y,f\}\in\cT_0\},
\;\;\;\;\;\Tma=\{\{\pi y, \pi f\}:\{y,f\}\in\tma\} \nonumber
\end{gather}
linear relations $\Tmi, \; T_a, \; T_0$ and $\Tma$ in the Hilbert space $\LI$.

As was shown in \cite{Orc} (see also \cite{LesMal03,BHSW10}) in the case of the
\emph{definite system} \eqref{1.8} $T_0$ is a symmetric linear relation in
$\LI$,  $\Tmi$ is closure  of $T_0$ and $\Tma=\Tmi^*(=T_0^*)$; moreover, if the
endpoint $a$ is regular then $\Tmi=T_a$. In view of this assertion $\Tmi$ and
$\Tma$ are called minimal and maximal relations respectively, which is in full
accord with similar definition of minimal and maximal operator for an ordinary
differential expression \cite{Nai}. At once  certain difficulties arise in the
case of an \emph{indefinite system} \eqref{1.8}, which can be explained as
follows. In the  definite case the quotient mapping $\pi$ isomorphically maps
$\dom\tma$ onto $\dom\Tma$, which enables one to identify in fact the relations
$\tma$  and $\Tma$. If the system is indefinite, then the mapping
$\pi\up\dom\tma$ has as nontrivial kernel the null manifold $\cN$, so that the
immediate identifying of  $\tma$  and $\Tma$ becomes impossible.

The above difficulties were partially overcome in the papers by Kac
\cite{Kac83,Kac84} (the case $\dim\bH=2$) and Lesch and Malamud \cite{LesMal03}
(the case $\dim\bH=n<\infty$), where general (not necessarily definite) systems
were studied. In these papers first the equality $T_0^*=\Tma$ is proved and
then the minimal relation is defined as  closure of $T_0$.

In the present paper we show that for the general system \eqref{1.8} the
minimal relation in $\LI$ can be also defined by the first equality in
\eqref{1.14} with $\tmi$ in the form \eqref{1.12}. Moreover in the case of the
regular endpoint $a$ the minimal relation coincides with the relation $T_a$
defined by \eqref{1.13} and the second equality in \eqref{1.14}. Observe also
that $\cT_a\subset \tmi$ and an interesting in our opinion  fact is that
generally speaking  $\cT_a\neq\tmi$ (for more details see Proposition
\ref{pr4.11a} and Example \ref{ex4.11b}).

Next assume that $a$ is  a regular endpoint for the  canonical system
\eqref{1.8},
\begin{gather*}
\nu_+:=\dim\ker (i J-I)(=\dim H), \qquad \nu_-:=\dim\ker (i J+I)(=\dim (\hat
H\oplus H))
\end{gather*}
and let $\nu_{b+} $ and $\nu_{b-} $ be indices of inertia of the skew-Hermitian
form $[y,z]_b$ (for simplicity assume that $\nu_{b+}\geq \nu_{b-}$). The
equality $\Tmi=T_a$ enables us to describe all proper extensions  of $\Tmi$ in
terms of boundary conditions. For this aim we use a special boundary relation
for $\Tma$ which we call decomposing. This boundary relation is defined as
follows.

Let $\cH_b$ and $\hat\cH_b$ be finite-dimensional  Hilbert spaces  and let
\begin{gather}\label{1.15}
\G_b=(\G_{0b}:\,  \hat\G_b:\,  \G_{1b})^\top:\dom\tma\to
\cH_b\oplus\hat\cH_b\oplus \cH_b
\end{gather}
be a surjective linear map such that
\begin{gather}\label{1.16}
[y,z]_b=i (\hat\G_b y, \hat\G_b z)-(\G_{1b}y,\G_{0b}z)+(\G_{0b}y,\G_{1b}z),
\quad y,z \in \dom\tma.
\end{gather}
(it is not difficult to prove the existence of such a map $\G_b$). Moreover,
for each function $y\in\dom\tma$ let
\begin{gather}\label{1.17}
y(t)=\{y_0(t),\, \hat y(t), \, y_1(t)  \}
\end{gather}
be the representation of  $y(t)$ in accordance with the decomposition
\eqref{1.9}. Then the decomposing boundary relation $\G:(\LI)^2\to\HH$ for
$\Tma$ is of the form
\begin{gather}\label{1.18}
\G=\lb \lb {\pi y\choose \pi f}, {\G_0' y\choose \G_1'y } \rb
:\{y,f\}\in\tma\rb,
\end{gather}
where $\cH_0$ and $\cH_1$ are Hilbert spaces defined by means of $\bH, \cH_b$
and $\hat \cH_b$ and $\G_j': \dom\tma\to \cH_j, \; j\in\{0,1\}$ are linear maps
constructed with the aid of $y(a)$ and the operators from \eqref{1.15}. If
$\Tmi$ has equal deficiency indices $n_+(\Tmi)=n_-(\Tmi)$, then
$\cH_0=\cH_1=H\oplus\hat H\oplus \cH_b$ and the decomposing boundary relation
\eqref{1.18} can be written as
\begin{gather}\label{1.19}
\G=\lb \lb {\pi y\choose \pi f}, {\{ y_0(a),\, \tfrac i {\sqrt 2}\ (\hat
y(a)-\hat\G_b y),\,\G_{0b}y \}\choose \{ y_1(a),\, \tfrac 1 {\sqrt 2} (\hat
y(a)+\hat\G_b y),\,-\G_{1b}y \} } \rb :\{y,f\}\in\tma\rb.
\end{gather}
In the case of the regular system one can put in \eqref{1.19} $\G_{0b}y=y_0(b),
\; \G_{1b}y=y_1(b)$ and $\hat\G_b y=\hat y(b)$. If $b$ is not regular, then
$\G_b y $ can be represented by means of certain limits at the point $b$
associated with the function $y\in\dom\tma$ (for more details see Remark
\ref{rem5.2} ). Therefore the decomposing boundary relation $\G$ is given by
\eqref{1.19} in terms of boundary values of the function $y\in\tma$ at the
endpoint $a$ (regular value) and $b$ (singular value), which is of importance
in our considerations of canonical systems.

Recall \cite{KogRof75} that the formal deficiency indices of the system
\eqref{1.8} are defined via
\begin{gather*}
N_\pm=\dim \{y\in\lI:\, J y'(t)-B(t)y(t)=\l\D(t)y(t) \;\;\text{a.e. on}\;\;
\cI\}, \quad \l\in\bC_\pm.
\end{gather*}
As was shown in \cite{LesMal03} the relations
\begin{gather*}
N_+=n_+(\Tmi)+k_\cN, \qquad N_-=n_-(\Tmi)+k_\cN
\end{gather*}
hold with $k_\cN=\dim\cN$. In the present paper by using just a fact of
existence of a decomposing boundary relation we prove the equalities
\begin{gather}\label{1.20}
N_+=\nu_+ + \nu_{b+}, \qquad N_-=\nu_- + \nu_{b-}.
\end{gather}
Formula \eqref{1.20} yields the known estimates $\nu_\pm\leq N_\pm\leq \dim\bH$
obtained by analytic methods in\cite{Atk,KogRof75}. Observe also that in a
somewhat different way the equalities \eqref{1.20} were proved for definite
systems  in \cite[Lemma 4.15]{BHSW10}.

Existence of the nontrivial multivalued part $\mul\G$ of the decomposing
boundary relation \eqref{1.18} is caused by the nontrivial null manifold $\cN$,
which can be seen from the equalities
\begin{gather}\label{1.21a}
\mul\G=\{\{\G_0'y,\G_1'y\}:\, y\in\cN\}, \qquad \dim (\mul\G)=k_\cN(=\dim\cN).
\end{gather}
Formula \eqref{1.21a} implies that for the definite canonical system
\eqref{1.8} the decomposing boundary relation turns into the decomposing
boundary triplet $\Pi=\bta$ for $\Tma$. In the case $n_+(\Tmi)=n_-(\Tmi)$ this
triplet is of the form $\Pi=\bt$ with the boundary Hilbert space
$\cH=H\oplus\hat H \oplus \cH_b$ and the operators $\G_j$ given by
\begin{gather}
\G_0 \{\wt y, \wt f\} =\{ y_0(a),\, \tfrac i {\sqrt 2} (\hat y(a)-\hat\G_b
y),\,\G_{0b}y \} (\in H\oplus\hat H \oplus \cH_b),\qquad\qquad\qquad\quad
\label{1.22}\\
\G_1 \{\wt y, \wt f\} = \{ y_1(a),\, \tfrac 1 {\sqrt 2} (\hat y(a)+\hat\G_b
y),\,-\G_{1b}y \}(\in H\oplus\hat H \oplus \cH_b), \quad \{\wt y, \wt f\} \in
\Tma \label{1.23}
\end{gather}
(here $\G_{0b}, \;\G_{1b}$ and $\hat\G_b$ are taken from \eqref{1.15}). In the
case of the regular system one can put $\cH= H\oplus\hat H \oplus H$ and
\begin{gather}
\G_0 \{\wt y, \wt f\}=\{ y_0(a),\, \tfrac i {\sqrt 2} (\hat y(a)-\hat
y(b)),\, y_0(b) \}(\in H\oplus\hat H \oplus H),\qquad\qquad\qquad\quad\label{1.24} \\
\G_1 \{\wt y, \wt f\}= \{ y_1(a),\, \tfrac 1 {\sqrt 2}(\hat y(a)+\hat
y(b)),\,-y_1(b) \}(\in H\oplus\hat H \oplus H),\quad \{\wt y, \wt f\} \in \Tma.
\label{1.25}
\end{gather}
The boundary triplet $\bt$ defined via \eqref{1.24} and \eqref{1.25} is similar
to that introduced, in fact,  by Rofe-Beketov \cite{Rof69} for regular
differential operators of a higher order. Observe also that other constructions
of a boundary triplet for $\Tma$ in the case of the definite system \eqref{1.8}
can be found in \cite{BHSW10}.

 The decomposing boundary triplet \eqref{1.22}, \eqref{1.23} enables us to describe
maximal dissipative, maximal accumulative and self-adjoint boundary conditions,
which define in terms of boundary values the extensions $\wt A\supset \Tmi$ of
the corresponding class. As a consequence we obtain the known description of
self-adjoint boundary conditions, given in \cite{Atk,GK,Orc} for regular
definite systems \eqref{1.8} and in \cite{Kra89} for definite Hamiltonian
systems with the regular endpoint $a$.

Finally by using the concept of a decomposing boundary triplet we examine
separated boundary conditions of various classes. Recall that self-adjoint
separated boundary conditions for definite Hamiltonian systems were studied
with the aid of analytic methods by many authors (see \cite{HinSch93,Kra89} and
references therein).  In the present paper we show that self-adjoint separated
boundary conditions for  the definite canonical system \eqref{1.8} exist if and
only if this system is Hamiltonian. Moreover, for the Hamiltonian system the
decomposing boundary triplet $\bt$ for $\Tma$ takes the form
\begin{gather*}
\G_0 \{\wt y, \wt f\} =\{ y_0(a),\, \G_{0b}y \} (\in H\oplus \cH_b),\qquad \G_1
\{\wt y, \wt f\} = \{ y_1(a),\, -\G_{1b}y \}(\in H \oplus \cH_b), \quad \{\wt
y, \wt f\} \in \Tma
\end{gather*}
and the general form of self-adjoint separated  boundary conditions is
\begin{gather}\label{1.26}
\wt A=\{\{\wt y, \wt f\}:\, N_{0a}y_0(a)+N_{1a}y_1(a)=0, \;\;
N_{0b}\G_{0b}y+N_{1b}\G_{1b}y=0 \},
\end{gather}
where the operators $N_{0a},\; N_{1a}$ and $N_{0b},\; N_{1b}$ are entries of
the self-adjoint operator pairs $\{(N_{0a}, N_{1a})\}$ and $\{(N_{0b},
N_{1b})\}$. These results are generalizations of those obtained by Rofe-Beketov
in \cite{Rof69} for regular differential operators. Moreover,  formula
\eqref{1.26} includes as a particular case the results on self-adjoint
separated  boundary conditions in \cite{HinSch93,Kra89}.

An interesting in our opinion fact is the existence of maximal dissipative an
accumulative separated boundary conditions for the not necessarily Hamiltonian
system \eqref{1.8} (in the paper we describe all these conditions). An
important subclass of maximal dissipative (accumulative) separated conditions
are those defined by a self-adjoint condition at the regular endpoint $a$ and
the maximal dissipative (accumulative) boundary condition at the singular
endpoint $b$. This subclass of boundary conditions may be useful in the theory
of not orthogonal spectral functions associated with the system \eqref{1.8} (we
are going to touch upon this subject elsewhere).

\section{Preliminaries}
\subsection{Linear relations}
The following notations will be used throughout the paper: $\gH$, $\cH$ denote
Hilbert spaces; $[\cH_1,\cH_2]$  is the set of all bounded linear operators
defined on $\cH_1$ with values in $\cH_2$; $[\cH]:=[\cH,\cH]$; $A\up \cL$ is
the restriction of an operator $A$ onto the linear manifold $\cL$; $P_\cL$ is
the orthogonal projector in $\gH$ onto the subspace $\cL\subset\gH$;
$\bC_+\,(\bC_-)$ is the upper (lower) half-plain  of the complex plain.

Recall that a linear relation $T$  from a linear space $L_0 $ to a linear space
$L_1$ is a linear manifold in the Cartesian product $L_0\times L_1$. It is
convenient to write $T: L_0\to L_1$ and interpret $T$ as a multi-valued linear
mapping from $L_0$ into $L_1$. If $L_0=L_1=:L$ one speaks of a linear relation
$T$ in $L$. For a linear relation $T:L_0\to L_1$ we denote by $\dom T,\,\ran T,
\,\ker T$ and $\mul T$  the domain,  range, kernel and the multivalued part of
$T$ respectively. The inverse $T^{-1}$ is a linear relation from $L_1$ to $L_0$
defined by $T^{-1}=\{\{f', f\}:\{f,f'\}\in T\}$.

Assume now that $\cH_0$ and $\cH_1$ are Hilbert spaces. Then the linear space
$\cH_0\times\cH_1$ with the inner product $(\{f,f'\}, \{g,g'\})_{\cH_0
\oplus\cH_1} =(f,g)_{\cH_0}+(f',g')_{\cH_1}$ is a Hilbert space
$\cH_0\oplus\cH_1$.  The set of all closed linear relations from $\cH_0$ to
$\cH_1$ (in $\cH$) will be denoted by $\C (\cH_0,\cH_1)$ ($\C(\cH)$). A closed
linear operator $T$ from $\cH_0$ to $\cH_1$  is identified  with its graph
$\text {gr}\, T\in\CA$. For a linear relation $T:\cH_0\to \cH_1$ we denote by
$T^*(\in\CB)$ the adjoint relation

In the case $T\in\CA$ we write $0\in \rho (T)$ if $\ker T=\{0\}$\ and\ $\ran
T=\cH_1$, or equivalently if $T^{-1}\in [\cH_1,\cH_0]$; $0\in \hat\rho (T)$\ \
if\ \  $\ker T=\{0\}$\ and\   $\ran (T)$ is a closed subspace in $\cH_1$. For a
linear relation $T\in \C(\cH)$ we denote by $\rho (T):=\{\l \in \bC:\ 0\in \rho
(T-\l)\}$ and $\hat\rho (T)=\{\l \in \bC:\ 0\in \hat\rho (T-\l)\}$ the
resolvent set and the set of regular type points  of $T$ respectively.

For a linear relation $T\in \C(\cH)$ and for any $\l \in \bC$ we let
\begin{equation*}
\gN_\l (T):=\ker (T^*-\l)\,(=\cH\ominus \ran (T-\ov\l)), \quad \hat
\gN_{\l}(T):=\{\{f,\l f\}:f\in \gN_{\l }(T)\}\subset T^*.
\end{equation*}
If $\ov\l\in\hat\rho (T)$, then $\gN_\l (T)$ is a defect subspace of  $T$.
Recall also the following definition.
\begin{definition}\label{def2.0}
A holomorphic operator function $\Phi (\cd):\bC\setminus\bR\to [\cH]$ is called
a Nevanlinna function  if $\im\, z\cd \im \Phi (z)\geq 0 $ and $\Phi ^*(z)=
\Phi (\ov z), \; z\in\bC\setminus\bR$.
\end{definition}
\subsection{Operator pairs }
Let $\cK,\cH_0,\cH_1$ be Hilbert spaces. A pair of operators $C_j\in
[\cH_j,\cK],  \;j\in\{0,1\}$  is called admissible if the range of the operator
\begin {equation}\label{2.1}
C=(C_0:\,C_1):\cH_0\oplus\cH_1\to\cK
\end{equation}
coincides with $\cK$. In the sequel all pairs \eqref{2.1} are admissible unless
otherwise stated.

 Two  pairs $(C_0^{(j)}:\,C_1^{(j)}):
\cH_0\oplus\cH_1\to\cK_j, \; j\in\{1,2\}$ will be called equivalent if
$C_0^{(2)}=XC_0^{(1)}$ and $C_1^{(2)}=XC_1^{(1)}$ with some isomorphism $X\in
[\cK_1,\cK_2]$.

It is clear that the set of all operator pairs \eqref{2.1}  falls into
nonintersecting classes of equivalent pairs. Moreover, the equality
\begin {equation}\label{2.2}
\t=\{(C_0,C_1);\cK\}:=\{\{h_0, h_1\}\in\cH_0\oplus\cH_1:\, C_0h_0+C_1h_1=0\}
\end{equation}
establishes a bijective correspondence between all linear relations  $\t\in\CA$
and all equivalence classes of  operator pairs \eqref{2.1}. Therefore in the
sequel  we identify (by means of \eqref{2.2}) a linear relation  $\t\in\CA$ and
the corresponding class of equivalent operator pairs $C_j\in [\cH_j,\cK], \;
j\in \{0,1\}$.

Next recall some results and definitions from our paper \cite{Mog06.1}.

Assume that $\cH_0$ is a Hilbert space, $\cH_1$ is a subspace in $\cH_0$,
$\cH_2:=\cH_0\ominus\cH_1$ and  $P_j$ is the orthoprojector in $\cH_0$ onto
$\cH_j,\; j\in\{1,2\}$. For an operator pair (linear relation)
$\t=\{(C_0,C_1);\cK\} (\in\CA)$ we let
\begin {gather*}
\wt S_\t:=2 \im (C_1C_{01}^*)-C_{02}C_{02}^*, \quad \wt S_\t\in [\cK],
\end{gather*}
where $C_{01}$ and $C_{02}$ are entries of the block representation
$C_0=(C_{02}:\, C_{01}):\cH_2\oplus\cH_1\to\cK$.
\begin{definition} \label{def2.1}$\,$ \cite{Mog06.1}
An operator pair (linear relation)  $\t=\{(C_0,C_1);\cK\}(\in\CA)$ belongs to
the class:

1) $Dis (\cH_0,\cH_1)$, if $\wt S_\t\geq 0$ and $0\in\rho (C_{01}-iC_1) $;

2) $Ac (\cH_0,\cH_1))$, if  $\wt S_\t \leq 0$ and $0\in\rho (C_0+iC_1P_1) $;

3) $Sym (\cH_0,\cH_1))$, if  $\wt S_\t = 0$ and $0\in\rho (C_{01}-iC_1) \cup
(C_0+iC_1P_1) $;

4))$Self (\cH_0,\cH_1)$, if $\t\in Dis (\cH_0,\cH_1)\cap Ac (\cH_0,\cH_1))$.
\end{definition}
Note that the inclusion $0\in\rho (C_{01}-iC_1)\cup \rho (C_0+iC_1P_1)$ implies
that $\ran (C_0:\, C_1)=\cK$. Therefore each of the above definitions
 1) -- 4)  gives an admissible operator
pair $(C_0:\, C_1)$.

If $\t\in Dis (\cH_0,\cH_1)$ ($\t\in Ac (\cH_0,\cH_1)$), then
$\dim\cK=\dim\cH_1$ (resp. $\dim\cK=\dim\cH_0$). Therefore for an operator pair
$\t\in Dis (\cH_0,\cH_1)\;(\t\in Ac (\cH_0,\cH_1))$ one can put $\cK=\cH_1$
(resp. $\cK=\cH_0$).

In the case $\cH_0=\cH_1:=\cH$ we let $Dis (\cH)=Dis (\cH,\cH)$ and similarly
the classes $ Ac (\cH),\; Sym (\cH)$ and $ Self (\cH)$ are defined. In view of
Definition \ref{def2.1} for each operator pair (linear relation)
$\t=\{(C_0,C_1);\cK\}(\in\C (\cH))$ the following equivalences hold:
\begin{gather}
\t\in Dis (\cH) \iff (\im (C_1C_0^*)\geq 0 \;\;\text{and}\;\; 0\in\rho
(C_0-iC_1) )\label{2.2.1}\\
\t\in Ac (\cH) \iff (\im (C_1C_0^*)\leq 0 \;\;\text{and}\;\; 0\in\rho
(C_0+iC_1) )\label{2.2.2}\\
\t\in Sym (\cH) \iff (\im (C_1C_0^*)= 0 \;\;\text{and}\;\; 0\in\rho
(C_0-iC_1)\cup\rho (C_0+iC_1) )\label{2.2.3}\\
\t\in Self (\cH) \iff (\im (C_1C_0^*)= 0 \;\;\text{and}\;\; 0\in\rho
(C_0-iC_1)\cap\rho (C_0+iC_1) )\label{2.2.4}
\end{gather}
Moreover, the classes $Dis (\cH), \;Ac (\cH),\; Sym (\cH)$ and $ Self (\cH)$
coincide with the known classes of all maximal dissipative, maximal
accumulative, maximal symmetric and self-adjoint linear relations in $\cH$
respectively.

The following  proposition  is immediate from Definition \ref{def2.1}.
\begin{proposition}\label{pr2.2}
1) In the case $\dim \cH_0<\infty$ the class $Self (\cH_0,\cH_1)$ is not empty
if and only if $\cH_0=\cH_1=:\cH$.

2) Let $\dim\cH<\infty $ and let $\t=\{(C_0,C_1);\cK\}(\in \C (\cH))$ be an
admissible operator pair such that $\dim\cK=\dim\cH$. Then the following
equivalences hold
\begin{gather*}
\t\in Dis (\cH) \Leftrightarrow (\im (C_1C_0^*)\geq 0, \quad \t\in Ac (\cH)
\Leftrightarrow (\im (C_1C_0^*)\leq 0, \quad \t\in Self (\cH) \Leftrightarrow
(\im (C_1C_0^*)= 0.
\end{gather*}
\end{proposition}
\subsection{Boundary triplets and Weyl functions}
Let $A\in\C (\gH)$ be a closed  symmetric linear relation in the Hilbert space
$\gH$ and let $n_\pm (A):=\dim \gN_\l(A), \; \l\in\bC_\pm$ be deficiency
indices of $A$. Denote by $\exa$ the set of all proper extensions of $A$, i.e.,
the set of all relations  $\wt A\in \C (\gH)$  such that $A\subset\wt A\subset
A^*$.

 Next assume that $\cH_0$ is a Hilbert space,  $\cH_1$ is a subspace
in $\cH_0$ and   $\cH_2:=\cH_0\ominus\cH_1$, so that $\cH_0=\cH_1\oplus\cH_2$.
Denote by $P_j$ the orthoprojector  in $\cH_0$ onto $\cH_j,\; j\in\{1,2\} $.
\begin{definition}\label{def2.3}$\,$\cite{Mog06.2}
A collection $\Pi=\bta$, where $\G_j: A^*\to \cH_j, \; j\in\{0,1\}$ are linear
mappings, is called a boundary triplet for $A^*$, if the mapping $\G :\hat f\to
\{\G_0 \hat f, \G_1 \hat f\}$ from $A^*$ into $\cH_0\oplus\cH_1$ is surjective
and the following Green's identity
\begin {equation}\label{2.3}
(f',g)-(f,g')=(\G_1  \hat f,\G_0 \hat g)_{\cH_0}- (\G_0 \hat f,\G_1 \hat
g)_{\cH_0}+i(P_2\G_0 \hat f,P_2\G_0 \hat g)_{\cH_2}
\end{equation}
 holds for all $\hat
f=\{f,f'\}, \; \hat g=\{g,g'\}\in A^*$.
\end{definition}
In the following propositions some properties  of boundary triplets are
specified (see \cite{Mog06.2}).
\begin{proposition}\label{pr2.4}
If $\Pi=\bta$ is a boundary triplet for   $A^*$, then
\begin {equation}\label{2.4}
\dim \cH_1=n_-(A)\leq n_+(A)=\dim \cH_0.
\end{equation}
Conversely for any symmetric linear relation  $A\in \C (\gH)$ with $n_-(A)\leq
n_+(A)$ there exists a boundary triplet for $A^*$.
\end{proposition}
\begin{proposition}\label{pr2.5}
Let $\Pi=\bta$ be a boundary triplet for  $A^*$. Then:

1) $\ker \G_0\cap\ker\G_1=A$ and $\G_j$ is a bounded operator from $A^*$ into
$\cH_j, \;  j\in\{0,1\}$;

2) the set of all proper extensions $\wt A\in\exa$ is parameterized by linear
relations (operator pairs) $\t=\{(C_0,C_1);\cK\}$. More precisely, the mapping
 \begin {equation}\label{2.5}
\t\to  A_\t :=\{ \hat f\in A^*:\{\G_0  \hat f,\G_1 \hat f \}\in \t\}
 \end{equation}
establishes a bijective correspondence between the linear relations  $\t\in\CA$
and the extensions $ \wt A= A_\t\in\exa$. If $\t$ is given as an operator pair
$\t=\{(C_0,C_1);\cK \}$, then $A_\t$ can be represented in the form of an
abstract boundary condition:
\begin {equation}\label{2.5a}
A_\t=\{\hat f\in A^*: C_0\G_0 \hat f+C_1\G_1\hat  f=0 \}
\end{equation}

3) the extension $ A_\t$ is maximal dissipative, maximal
 accumulative, maximal symmetric or self-adjoint if and only if $\t$ belongs to
 the class $Dis,\; Ac,\; Sym$ or $Self (\cH_0,\cH_1)$ respectively;

 4) The equality
\begin {equation}\label{2.6}
A_0:=\ker \G_0=\{\hat f\in A^*:\G_0 \hat f=0\}
 \end{equation}
 defines the maximal symmetric extension $A_0\in\exa$ such that $n_-(A_0)=0$.
\end{proposition}
It turns out that for every $\l\in\bC_+\;(z\in\bC_-)$ the map $\G_0\up \hat
\gN_\l (A)$ (resp $P_1\G_0\up \hat \gN_z (A)$) is an isomorphism. This makes it
possible to introduce the $\g$-fields $\g_{\Pi +}(\cdot):\Bbb
C_+\to[\cH_0,\gH], \; \; \g_{\Pi -}(\cdot):\Bbb C_-\to[\cH_1,\gH]$ and the Weyl
functions $M_{\Pi +}(\cdot):\bC_+\to [\cH_0,\cH_1], \;\; M_{\Pi
-}(\cdot):\bC_-\to [\cH_1,\cH_0]$ by
\begin{gather}
\g_{\Pi +} (\l)=\pi_1(\G_0\up\hat \gN_\l (A))^{-1}, \;\;\l\in\Bbb C_+;\quad
\g_{\Pi -} (z)=\pi_1(P_1\G_0\up\hat\gN_z (A))^{-1}, \;\; z\in\Bbb C_-,
\label{2.7}\\
\G_1 \up \hat\gN_\l (A)=M_{\Pi +}(\l)\G_0 \up \hat\gN_\l (A),\quad \l\in\bC_+,
\label{2.8}\\
(\G_1+iP_2\G_0)\up \hat\gN_z (A)=M_{\Pi -}(z)P_1\G_0 \up \hat\gN_z (A),\quad
z\in \bC_-.\label{2.9}
\end{gather}
(here $\pi_1$ is the orthoprojection in $\gH\oplus\gH$ onto $\gH\oplus \{0\}$).
According to \cite{Mog06.2} all functions $\g_{\Pi\pm}$ and $M_{\Pi\pm} $ are
holomorphic on their domains and $(M_{\Pi +}(\l))^*=M_{\Pi-}(\ov \l),
\;\l\in\bC_+$.
\begin{remark}\label{rem2.6}
In the case $\cH_0=\cH_1:=\cH$ Definition \ref{def2.3} coincides with that of
the boundary triplet (boundary value space) $\Pi=\{\cH,\G_0,\G_1\}$ for $A^*$
given in \cite{GorGor}. For such a triplet $n_+(A)=n_-(A)=\dim \cH$,
$\,A_0(=\ker \G_0 )$ is a self-adjoint extension of $A$ and the relations
\begin {equation}\label{2.10}
\g_\Pi(\l)=\pi_1(\G_0\up\hat\gN_\l(A))^{-1}, \qquad \G_1\up\hat
\gN_\l(A)=M_\Pi(\l)\G_0 \up\hat\gN_\l(A), \qquad \l\in\rho (A_0)
\end{equation}
define the $\g$-field $\g_\Pi(\cd):\rho (A_0)\to [\cH,\gH]$
 and the Weyl function $M_\Pi(\cd):\rho (A_0)\to [\cH]$ \cite{DM91} associated
 with operator functions  \eqref{2.7}--\eqref{2.9} via $\g_\Pi(\l)=\g_{\Pi\pm}
(\l)$
  and $M_\Pi(\l)=M_{\Pi\pm}(\l), \;\l\in\bC_\pm $.
\end{remark}
\section{Boundary relations and their Weyl families}
Let $\gH $ and $\cH_0$ be Hilbert spaces, let $\cH_1$ be a subspace in $\cH_0$,
let $\cH_0=\cH_1\oplus\cH_2$ be the corresponding orthogonal decomposition of
$\cH_0$ with $\cH_2:=\cH_0\ominus\cH_1$ and let $P_j$ be the orthoprojector in
$\cH_0$ onto $\cH_j, \; j\in \{0,1\}$. In the sequel we deal with linear
relations from $\gH^2$ into $\cH_0\oplus\cH_1$. If $\G$ is such a relation,
then an element $\hat\f\in\G$ will be denoted by $\hat\f =\{\hat f,\hat h\}$,
where $\hat f=\{f,f'\}\in\gH^2\; (f,\,f'\in \gH)$ and $\hat h=\{h_0,h_1\}\in
\cH_0\oplus\cH_1\; (h_0\in\cH_0,\; h_1\in\cH_1)$. In this case it will be
convenient to write
\begin {equation}\label{3.1a}
\hat\f=\{\hat f,\hat h\}=\lb\hat f, {h_0 \choose h_1} \rb=\lb
\begin{pmatrix} f\cr f' \end{pmatrix},
\begin{pmatrix} h_0\cr h_1 \end{pmatrix}\rb.
\end{equation}
If in addition $\cH_j$ is decomposed as $\cH_j=\cH_{j1}\oplus\cH_{j2}\oplus
\dots  \cH_{j,n_j}, \; j\in\{0,1\}$, then the equality \eqref{3.1a} will be
also written as
\begin {equation*}
\hat\f=\lb\hat f, {\{h_{01}, h_{02}, \dots, h_{0,n_0}\}\choose \{h_{11},
h_{12}, \dots, h_{1,n_1}\}} \rb,
\end{equation*}
where $h_{0k}=P_{\cH_{0k}}h_0$ and $h_{1k}=P_{\cH_{1k}}h_1$ are components of
$h_0$ and $h_1$ respectively.

For a linear relation $\BR$ its multivalued part is a linear relation from
$\cH_0$ into $\cH_1$ given by
\begin {equation*}
 \mul\G =\lb \{h_0,h_1\}\in\cH_0\oplus\cH_1: \lb 0,{h_0\choose h_1}\rb \in\G \rb.
\end{equation*}
Using $\mul\G$ we define linear manifolds $\cK_\G'$ and $\cK_G''$ in $\cH_1$
via
\begin {gather}
\cK_\G'=P_1\,\dom (\mul\G)=\lb k'\in\cH_1: \lb 0,{k'+h_2\choose h_1} \rb \in\G
\;\;\text{for some}\;\; h_2\in\cH_2 \;\; \text{and}\; \; h_1\in\cH_1 \rb,
\label{3.3}\\
\cK_\G''=\mul (\mul\G)=\lb k''\in\cH_1:\lb 0,{0\choose
k''}\rb\in\G\rb.\label{3.4}
\end{gather}
Next introduce the signature operators
\begin {equation}\label{3.5}
J_\gH=\begin{pmatrix} 0 & -i I_\gH \cr iI_\gH & 0\end{pmatrix}:\gH\oplus\gH\to
\gH\oplus\gH, \;\;\;\; J_{01}=\begin{pmatrix} P_2 & -i I_{\cH_1} \cr iP_1 &
0\end{pmatrix}:\cH_0\oplus\cH_1\to \cH_0\oplus\cH_1
\end{equation}
and denote by $(\gH^2,J_\gH)$ and $(\HH,J_{01})$ the corresponding Krein
spaces. Recall \cite{Shm76} that a linear relation $\BR$ is called an isometric
relation from $(\gH^2,J_\gH)$ into $(\HH,J_{01})$ if
\begin {equation}\label{3.6}
(J_\gH \hat f,\hat g)_{\gH^2}=(J_{01} \hat h, \hat x)_{\cH_0\oplus \cH_1},
\quad \{\hat f, \hat h\}, \;\{\hat g, \hat x\}\in\G
\end{equation}
or, equivalently, if the identity
\begin {equation}\label{3.7}
(f',g)_\gH -  (f,g')_\gH= (h_1,x_0)_{\cH_0} - (h_0,x_1)_{\cH_0}+i (P_2h_0,P_2
x_0)_{\cH_2}
\end{equation}
holds for every $\{{f \choose f'}, {h_0 \choose h_1}\}, \;\{{g \choose g'},
{x_0 \choose x_1}\} \in \G $.
\begin{definition}\label{def3.2}
Let $A$ be a closed symmetric linear relation in $\gH$, let $\cH_0$ be a
Hilbert space and let $\cH_1$ be a subspace in $\cH_0$. A linear relation $\BR$
is called a boundary relation for $A^*$ if:

1) $\dom \G$ is dense in $A^*$ and $\G$ is an isometric relation from
$(\gH^2,J_\gH)$ into  $(\HH,J_{01})$, i.e., the abstract Green's identity
\eqref{3.7} holds;

2) if  $\hat\f=\{{g\choose g'}, {x_0\choose x_1}\}\in \gH^2\oplus
(\cH_0\oplus\cH_1)$ satisfies \eqref{3.7} for every $\{{f\choose f'},
{h_0\choose h_1}\}\in\G$, then $\hat\f\in\G$.
\end{definition}
The conditions 1) and 2) of Definition \ref{def3.2} imply that the boundary
relation $\G$ is a unitary relation from $(\gH^2,J_\gH)$ to $(\HH, J_{01})$
\cite{Shm76}. Therefore $\G$ is closed and $\ker\G=A$.
\begin{definition}\label{def3.2a}
The families of linear relations $M_+(\l):\cH_0\to\cH_1, \; \l\in\bC_+$ and
$M_-(z):\cH_1\to\cH_0, \; z\in\bC_-$ given by
\begin{gather}
M_+(\l)=\lb \{h_0, h_1\}\in \HH: \lb {f\choose \l f}, {h_0\choose h_1}\rb
\in\G \;\;\text{for some} \;\; f\in\gH \rb, \;\;\l\in\bC_+; \label{3.7.1}\\
M_-(z)=\lb \{P_1h_0, h_1+iP_2 h_0\}: \lb {f\choose z f}, {h_0\choose h_1}\rb
\in\G \;\;\text{for some} \;\; f\in\gH \rb, \;\; z\in\bC_- \label{3.7.2}
\end{gather}
are called the Weyl families corresponding to the boundary relation $\BR$ for
$A^*$.

If $M_+(\cd)$ (resp. $M_-(\cd)$) is operator-valued, it is called the Weyl
function corresponding to the boundary relation $\G$.
\end{definition}
In the sequel we deal with boundary relations of the special form introduced in
the following proposition.
\begin{proposition}\label{pr3.3}
Assume that $\Pi=\{\cK_0\oplus\cK_1,G_0,G_1\}$ is a boundary triplet for $A^*$
(see Definition \ref{def2.3}), $\cK_2:=\cK_0\ominus\cK_1$, $\cK'$ and $\cK''$
are Hilbert spaces and
\begin {equation}\label{3.7a}
\cH_1:=\cK_1\oplus \cK'\oplus\cK'', \qquad \cH_0:=\cK_0\oplus \cK'\oplus\cK''
(=\cK_2\oplus\cH_1).
\end{equation}
Moreover, let $F_0\in [\cK',\cK_0], \; F'\in [\cK']$, let
\begin {equation}\label{3.8}
F_0=(F_2\;\; F_1)^\top : \cK'\to \cK_2\oplus \cK_1
\end{equation}
be the block representation of $F_0$ and let
\begin {equation}\label{3.9}
F'-(F')^*+i F_2^*F_2=0.
\end{equation}
Then the equality
\begin {equation}\label{3.10}
\G=\lb \lb \hat f, {\{G_0\hat f-iF_2k',\, k',\, 0\}\}\choose \{G_1\hat f +F_1
k',\, F_0^* G_0\hat f +F'k',\, k''\}} \rb :\hat f\in A^*,\; k'\in\cK', \;
k''\in\cK'' \rb
\end{equation}
defines the boundary relation $\BR$ for $A^*$ such that $\cK_\G'=\cK'$ and
$\cK_\G''=\cK''$.
\end{proposition}
\begin{proof}
It is easily seen that the following assertion (a) is valid:

(a) an element $\hat\f=\{\hat g, {x_0\choose x_1}\}\in \gH^2\oplus
(\cH_0\oplus\cH_1)$ with
\begin {equation}\label{3.11}
\hat g=\{g,g'\}\in\gH^2, \quad x_0=\{m_0,\,x_0', \, x_0''\}\in\cK_0\oplus
\cK'\oplus\cK'', \quad x_1=\{m_1,\, x_1',\, x_1''\}\in\cK_1\oplus
\cK'\oplus\cK''
\end{equation}
satisfies the identity \eqref{3.7} for every $\{{f\choose f'}, {h_0\choose
h_1}\}\in\G$ if and only if $x_0''=0$ and  the following equalities hold
\begin{gather}
(f',g)-(f,g')=\qquad\qquad\qquad\qquad\qquad\qquad\qquad\qquad\qquad
\qquad\qquad\qquad\qquad\qquad
\qquad\qquad\label{3.12} \\
=(G_1 \hat f, m_0)_{\cK_0}+(F_0^* G_0\hat f, x_0')_{\cK'}-(G_0 \hat f,
m_1)_{\cK_0}+i(P_{\cK_2} G_0\hat f,P_{\cK_2} m_0)_{\cK_2}, \;\;\; \hat f =
\{f,f'\} \in A^*; \nonumber\\
(F_0k', m_0)_{\cK_0}+(F'k',x_0')_{\cK'}-(k',x_1')_{\cK'}=0, \quad
k'\in\cK'.\label{3.13}
\end{gather}
Let $\{\hat g, {x_0\choose x_1}\}\in\G$, so that $\hat g\in A^*$ and $x_0, x_1$
are given by \eqref{3.11} with
\begin {equation*}
m_0=G_0 \hat g-i F_2 x_0', \quad x_0''=0, \quad m_1=G_1\hat g +F_1x_0', \quad
x_1'=F_0^* G_0\hat g+F' x_0'.
\end{equation*}
Then substitution of such $m_0, \; m_1$ and $x_1'$ into \eqref{3.12},
\eqref{3.13} and the immediate calculation with taking \eqref{3.9} and
\eqref{2.3} into account show that the equalities \eqref{3.12} and \eqref{3.13}
are satisfied. Therefore by assertion (a) the identity \eqref{3.7} holds for
every $\{{f \choose f'}, {h_0 \choose h_1}\}, \;\{{g \choose g'}, {x_0 \choose
x_1}\} \in \G $.

Assume now that an element $\{\hat g, {x_0\choose x_1}\}\in \gH^2\oplus
(\cH_0\oplus\cH_1)$ given by \eqref{3.11} satisfies the identity \eqref{3.7}
for every $\{{f\choose f'}, {h_0\choose h_1}\}\in\G$. Then by assertion (a)
$x_0''=0$ and  the equalities \eqref{3.12}, \eqref{3.13} are fulfilled.

If $\hat f=\{f,f'\}\in A$, then $G_0\hat f=G_1\hat f=0$ and by \eqref{3.12}
$(f',g)-(f,g')=0$. This implies that $\hat g\in A^*$. Next, in view of
\eqref{3.8} $F_0^* G_0\hat f =F_2^* P_{\cK_2}G_0\hat f+F_1^* P_{\cK_1}G_0\hat f
$ and the equality \eqref{3.12} can be written as
\begin{gather}
(f',g)-(f,g')=\qquad\qquad\qquad\qquad\qquad\qquad\qquad\qquad\qquad
\qquad\qquad\qquad\qquad\qquad
\qquad\qquad\label{3.14}\\
=(G_1 \hat f, P_{\cK_1}m_0)-(P_{\cK_1} G_0\hat f, m_1-F_1 x_0')+i(P_{\cK_2}
G_0\hat f,P_{\cK_2} m_0+iF_2x_0'), \;\;\;\; \hat f = \{f,f'\} \in A^*.\nonumber
\end{gather}
Since the map $G=(G_0\;\;G_1)^\top$ is surjective, it follows from \eqref{3.14}
and  \eqref{2.3} that
\begin {equation*}
P_{\cK_1}m_0=P_{\cK_1}G_0\hat g, \quad m_1-F_1 x_0'=G_1\hat g,\quad P_{\cK_2}
m_0+iF_2x_0'=P_{\cK_2}G_0\hat g
\end{equation*}
and, consequently,
\begin {equation*}
m_0=G_0\hat g-iF_2x_0', \quad m_1=G_1\hat g+F_1 x_0'.
\end{equation*}
Moreover, by using first \eqref{3.13} and then \eqref{3.8}, \eqref{3.9} one
obtains
\begin {equation*}
x_1'=F_0^* m_0+(F')^* x_0'=F_0^* G_0\hat g+((F')^*-i F_0^* F_2)x_0'=F_0^*
G_0\hat g+ F'x_0'.
\end{equation*}
Thus $\{\hat g, {x_0\choose x_1}\}\in \G$ and the linear relation \eqref{3.10}
satisfies both conditions of Definition \ref{def3.2}.

Finally the equalities $\cK_\G'=\cK'$ and $\cK_\G''=\cK''$ are immediate from
\eqref{3.10}.
\end{proof}
\begin{proposition}\label{pr3.3a}
Let under the assumptions of Proposition \ref{pr3.3} $\cK''=\{0\}$, so that
\begin {equation}\label{3.14.1}
\cH_1:=\cK_1\oplus \cK', \qquad \cH_0:=\cK_0\oplus \cK' (=\cK_2\oplus\cH_1).
\end{equation}
and the boundary relation $\BR$ for $A^*$ is (c.f. \eqref{3.10})

\begin {equation}\label{3.14.2}
\G=\lb \lb \hat f, {\{G_0\hat f-iF_2k',\, k'\}\}\choose \{G_1\hat f +F_1 k',\,
F_0^* G_0\hat f +F'k'\}} \rb :\hat f\in A^*,\; k'\in\cK' \rb.
\end{equation}
Assume also that $\g_{\Pi\pm}(\cd)$ and $M_{\Pi\pm}(\cd)$ are the $\g$-fields
and Weyl functions corresponding to the boundary triplet $\Pi$ (see
\eqref{2.7}-\eqref{2.9}) and $\G_0:\gH^2\to \cH_0$ is the linear relation given
by
\begin {equation*}
\G_0=P_{\cH_0\oplus \{0\}}\G=\lb\{\hat f, h_0\}\in \gH^2\oplus\cH_0: \lb \hat
f, {h_0\choose h_1} \rb \in \G \;\; \text{for some}\;\; h_1\in\cH_1 \rb.
\end{equation*}
Then: 1) $\ker\G_0=\ker G_0$, so that the equality
\begin {equation}\label{3.14.3}
A_0:=\ker \G_0=\{\hat f\in A^*: \{\hat f, 0\}\in\G_0\}
\end{equation}
gives the maximal symmetric extension $A_0\in\exa$ with $n_-(A_0)=0 \,(\iff
\bC_+\subset \rho (A_0))$;

2) the equalities
\begin{gather}
\hat\g_{+} (\l)=(\G_0\up\hat \gN_\l (A))^{-1},\qquad \g_{+}
(\l)=\pi_1\hat\g_{+}
(\l),\;\; \;\;\l\in\Bbb C_+;\label{3.14.4}\\
 \hat\g_{-} (z)=(P_1\G_0\up\hat\gN_z
(A))^{-1}, \qquad  \g_-(z)=\pi_1 \hat\g_{-} (z),\;\;\;\; z\in\Bbb
C_-,\label{3.14.5}
\end{gather}
define the holomorphic operator functions ($\g$-fields) $\g_{+}(\cdot):\Bbb
C_+\to[\cH_0,\gH]$ and $\g_{-}(\cdot):\Bbb C_-\to[\cH_1,\gH]$. Moreover,
\begin {equation}\label{3.14.6}
\g_+(\l)=\g_{\Pi +}(\l)\,(I_{\cK_0}\;\;iF_2), \;\;\l\in\bC_+; \qquad
\g_-(z)=\g_{\Pi -}(z)P_{\cK_1}, \;\;z\in\bC_-
\end{equation}
and the following identities hold
\begin{gather}
\g_+ (\mu)=\g_+ (\l )+ (\mu-\l )(A_0-\mu)^{-1} \g _{+} (\l), \quad \l, \mu\in
\bC_+  \label{3.14.7}\\
 \gamma _{-} (\omega)=\gamma _{-} (z)+
(\omega- z)(A_0^*-\omega)^{-1} \gamma _{-} (z),  \quad z, \omega\in
\bC_-  \label{3.14.8}\\
\gamma _{-} (z)P_1=\gamma _+ (\l)+ (z-\l)(A_0^*-z)^{-1} \gamma _{+} (\l), \quad
\l\in\bC_+,\; z\in \bC_-. \label{3.14.9}
\end{gather}

3) the corresponding Weyl families are the holomorphic operator functions
$M_{+}(\cdot):\bC_+\to [\cH_0,\cH_1]$ and $ M_{-}(\cdot):\bC_-\to
[\cH_1,\cH_0]$ associated with $M_{\Pi\pm}(\cd)$ by
\begin {equation}\label{3.14.10}
M_+(\l)=\begin{pmatrix} M_{\Pi +}(\l) & F_1+i\,M_{\Pi +}(\l)F_2 \cr F_0^* &
(F')^*\end{pmatrix}: \cK_0\oplus\cK'\to \cK_1\oplus\cK', \;\;\;\l\in\bC_+
\end{equation}
and $M_-(z)=M_+^*(\ov z), \; z\in\bC_-$. Moreover, the block representations
\begin{gather}
M_+(\l)=(N_+(\l)\;\; M(\l)):\cK_2\oplus\cH_1 \to \cH_1, \quad
\l\in\bC_+\label{3.14.11}\\
M_-(z)=(N_-(z)\;\; M(z))^\top: \cH_1\to\cK_2\oplus\cH_1  , \quad
z\in\bC_-\label{3.14.12}\\
\cM(\l)=\begin{pmatrix} \tfrac i 2 I_{\cK_2} & 0 \cr  N_+(\l) &
M(\l)\end{pmatrix}:\cK_2\oplus\cH_1 \to \cK_2\oplus\cH_1, \quad
\l\in\bC_+\label{3.14.13}\\
\cM(z)=\begin{pmatrix} -\tfrac i 2 I_{\cK_2} & N_-(z) \cr 0 & M(z)
\end{pmatrix} :\cK_2\oplus\cH_1 \to \cK_2\oplus\cH_1, \quad
z\in\bC_-\label{3.14.14}
\end{gather}
induce the Nevanlinna operator functions $M(\l)\,(\in [\cH_1])$ and $\cM(\l)\,
(\in [\cH_0])$;

4) the following identity holds
\begin {equation}\label{3.14.15}
\cM(\mu)-\cM^*(\l)=(\mu-\ov\l) \g_+^*(\l)\g_+(\mu), \quad \mu,\l\in\bC_+.
\end{equation}
\end{proposition}
\begin {proof}
The statement 1) is immediate from \eqref{3.14.2} and Proposition \ref{pr2.5},
3).

2) It follows from \eqref{3.14.2} that
\begin{gather*}
\G_0\up \hat\gN_\l(A)=\{\{\hat f_\l, \{G_0\hat f_\l-i \,F_2k',\, k'\}\}:\hat
f_\l\in \hat\gN_\l(A), \; k'\in\cK'\}, \;\;\;\;\l\in\bC_+\\
P_1\G_0\up \hat\gN_z (A)=\{\{\hat f_z, \{P_1 G_0\hat f_z, \, k'\}\}:\hat f_z\in
\hat\gN_z(A), \; k'\in\cK'\}, \;\;\;\; z\in\bC_-.
\end{gather*}
This and \eqref{2.7} imply that the equalities \eqref{3.14.4} and
\eqref{3.14.5} correctly define the operator functions $\g_+(\cd)$ and
$\g_-(\cd)$ such that \eqref{3.14.6} holds.

Next, one can prove the identities \eqref{3.14.7} - \eqref{3.14.9} in the same
way as in \cite [Proposition 3.15]{Mog06.2}. These identities  show that the
functions $\g_\pm(\cd)$ are holomorphic on their domains.

3)-4). Combining \eqref{3.7.1} with \eqref{3.14.2} one obtains
\begin {equation*}
M_+(\l)=\lb {\{G_0\hat f_\l-iF_2k',\, k'\}\choose \{M_{\Pi+}(\l)G_0 \hat f_\l
+F_1 k',\, F_0^* G_0\hat f_\l +F'k'\}}:\hat f_\l\in \hat\gN_\l (A),\; k'\in\cK'
\rb .
\end{equation*}
Letting here $h_0=G_0\hat f_\l-iF_2k'$ and taking \eqref{3.9} into account we
get
\begin {equation*}
M_+(\l)=\lb {\{h_0,\, k'\}\choose \{M_{\Pi+}(\l)h_0 +(F_1+i M_{\Pi+}(\l)F_2
)k',\, F_0^* h_0 +(F')^* k'\}}:h_0\in\cH_0,\; k'\in\cK' \rb .
\end{equation*}
This equality is equivalent to \eqref{3.14.10}. The identity \eqref{3.14.15} is
proved in the same way as similar one in \cite [Proposition 3.17]{Mog06.2}.
Finally, \eqref{3.14.15} implies that $\cM(\cd) $ and $M(\cd)$ are Nevanlinna
operator functions.
\end{proof}
In the next proposition we show that under the condition $\dim \cH_0<\infty $
formula \eqref{3.10} gives the general form of a boundary relation $\BR$ for
$A^*$.
\begin{proposition}\label {pr3.4}
Let $A$ be a closed symmetric linear relation in $\gH$ and let $\BR$ be a
boundary relation for $A^*$ such that $\dim \cH_0<\infty $. Then:

1) $n_\pm(A)<\infty$ and $\dom\G=A^*$;

2) the subspaces $\cK_\G'$ and $\cK_\G''$ (see \eqref{3.3} and \eqref{3.4}) are
mutually orthogonal and, consequently, the decompositions
\begin {equation}\label{3.14.16}
\cH_1:=\cK_1\oplus \cK_\G'\oplus\cK_\G'', \qquad \cH_0:=\cK_0\oplus
\cK_\G'\oplus\cK_\G''
\end{equation}
  hold with $\cK_1:=\cH_1\ominus (\cK_\G'\oplus\cK_\G'')$ and
  $\cK_0:=\cH_2\oplus \cK_1$;

3) for any $\{\hat f, {h_0\choose h_1}\}\in \G$ the inclusion $h_0\in
\cK_0\oplus \cK_\G'$ is valid;

4) for any $k'\in\cK_\G'$ there exists a unique triple of elements $ h_2\in
\cH_2, \; k_1\in\cK_1$ and $m'\in\cK_\G'$ such that
\begin {equation}\label{3.15}
\lb 0, {-i h_2+k' \choose k_1+m'} \rb\in\G.
 \end{equation}
Hence the relation \eqref{3.15} defines the operators $F_2 \in [\cK_\G',\cH_2],
\;F_1 \in [\cK_\G',\cK_1] $ and $F'\in [\cK_\G']$ via
\begin {equation}\label{3.16}
F_2 k'= h_2, \qquad F_1 k'=k_1, \qquad F'k'=m'\quad (k'\in\cK_\G').
\end{equation}
Moreover, these operators satisfy \eqref{3.9};

5) there exist linear maps $G_j:A^*\to \cK_j, \; j\in\{0,1\}$ such that
$\{\cK_0\oplus\cK_1, G_0, G_1\}$ is a boundary triplet for $A^*$ and $\G$
admits the representation \eqref{3.10} with $\cK'=\cK_\G', \;\cK''=\cK_\G''$
and the operator $F_0$ given by \eqref{3.8}.
\end{proposition}
\begin{proof}
1) Since $A=\ker \G$ and $\dim (\HH)<\infty$, it follows that $\dim (\dom\G /
A)<\infty$. Hence $\dom \G=\ov{\dom \G}=A^*$ and, consequently,
$n_+(A)+n_-(A)=\dim (A^*/ A) <\infty$.

2) - 3). Let $\{\hat f, {h_0\choose h_1}\}\in \G$ and $k''\in\cK_\G''$, so that
$\{0, {0\choose k''}\}\in \G$. Then by \eqref{3.7} $(h_0, k'')=0$, which
implies that $h_0\in \cH_0\ominus \cK_\G''$.

Assume now that $k'\in\cK_\G'$, so that $\{0, {k'+h_2\choose h_ 1}\}\in \G$
with some $h_2\in \cH_2 (\,\subset \cH_0 \ominus \cK_\G'' )$ and $h_1\in\cH_1$.
Then the above statement gives $k'+h_2\in \cH_0\ominus \cK_\G''$ and ,
consequently, $k'\in \cH_0\ominus \cK_\G''$. Therefore $\cK_\G'\perp \cK_\G''$.

4) Let $k'\in\cK_\G'$. Then according to \eqref{3.3} and the first equality in
\eqref{3.7a} $\{0, {k'-i h_2\choose k_1+m'+k''}\}\in \G$ with some $ h_2\in
\cH_2, \; k_1\in\cK_1,\; m'\in \cK_\G'$ and $k''\in\cK_\G''$, which in view of
\eqref{3.4} implies that $\left\{0, {k'-i h_2\choose k_1+m'}\right\}\in \G$.
Let us show that such $ h_2, \; k_1$ and $m'$ are unique for a given $k'$. If
$\left\{0, {k'-i\wt h_2\choose \wt k_1+ \wt m'}\right\}\in \G$ with some $\wt
h_2\in \cH_2, \;\wt k_1\in\cK_1$ and $\wt m'\in \cK_\G'$, then $\left\{0,
{i(h_2-\wt h_2)\choose (\wt k_1-k_1)+ (\wt m'-m')}\right\}\in \G$  and the
identity \eqref{3.7} yields $0=i|| h_2-\wt h_2||^2$. Therefore $ h_2-\wt h_2=0$
and, consequently, $(\wt k_1-k_1)+ (\wt m'-m')\in \cK_\G''$. Now the
decomposition \eqref{3.7a} yields $\wt k_1-k_1=0, \; \wt m'-m'=0$, so that $\wt
h_2=  h_2, \;\wt k_1=k_1$ and $\wt m'=m'$.

The equality \eqref{3.9} for operators $F_2$ and $F'$ is immediate from
identity \eqref{3.7} applied to $\left\{0, {-i F_2k'+k' \choose
F_1k'+F'k'}\right\}$ $\;\; (k'\in\cK_\G')$.

5) Combining \eqref{3.15} and \eqref{3.16} with \eqref{3.3} and \eqref{3.4} one
obtains
\begin {equation}\label{3.17}
\G_\infty:=\{0\}\oplus \mul\G=\lb \left\{0, {\{-i F_2k',\,k', \, 0 \} \choose
\{F_1k', \, F'k', \, k''\}}\right\}: k'\in\cK_\G', \; k''\in\cK_\G''  \rb.
\end{equation}
Let $T\subset \G$ be a linear relation from $\gH^2$ to $\HH$ given by
\begin {equation}\label{3.18}
T=\lb \lb \hat f, {h_0 \choose h_1}\rb\in\G: h_0\in \cK_0, \; h_1\in
\cK_0\oplus \cK_\G' \rb.
\end{equation}
Then in view of \eqref{3.17} and statement 3) $T\cap \G_\infty=\{0\}$ and the
decomposition
\begin {equation}\label{3.19}
\G=T\dotplus\G_\infty
\end{equation}
is valid. This and the equality $\dom \G=A^*$ imply that $T$ is the operator
with the domain $\dom T=A^*$. Moreover, applying the Green's identity
\eqref{3.7} to elements $\left\{0, {\{-i F_2k',\,k', 0 \} \choose \{F_1k', \,
F'k', \, 0\}}\right\}\in \G_\infty $ and $\left\{\hat f, {\{k_0,\,0,\, 0 \}
\choose \{k_1, \, m', \, 0\}}\right\}\in T $ one obtains
\begin {gather*}
0=(m',k')_{\cK'}-(k_0,F_1 k')_{\cK_0}-(k_0,F_2 k')_{\cK_0}=(
m',k')_{\cK'}-(k_0,F_0 k')_{\cK_0}, \;\;\;k'\in\cK_\G'.
\end{gather*}
Hence $m'=F_0^* k_0$ and formula \eqref{3.18} can be written as
\begin {equation}\label{3.20}
T=\lb \lb \hat f, {\{G_0\hat f,\, 0, \, 0\} \choose \{G_1\hat f,\, F_0^*G_0\hat
f, \, 0 \}}\rb: \hat f \in A^* \rb,
\end{equation}
where $G_0:=P_{\cK_0\oplus \{0\}}T$ and $G_1:=P_{\{0\}\oplus \cK_1}T$ are
linear maps from $A^*$ to $\cK_0$ and $\cK_1$ respectively. Combining
\eqref{3.19} with \eqref{3.17} and \eqref{3.20} we arrive at the representation
\eqref{3.10} of $\G$.

Now it remains to show that the operators $G_0$ and $G_1$ form the boundary
triplet $\Pi=\{\cK_0\oplus\cK_1, G_0, G_1\}$ for $A^*$. Applying the identity
\eqref{3.7} to elements of the linear relation $T$ (see \eqref{3.20}) one
obtains the Green's identity \eqref{2.3} for $G_0$ and $G_1$.

To prove surjectivity of the map $G=(G_0\;\;G_1)^\top$ assume that
$\{m_0,\,m_1\}\in \cK_0\oplus\cK_1$ and $(G_0\hat f, m_0)+(G_1\hat f, m_1)=0$
for all $\hat f\in A^*$. Then the element
\begin {equation}\label{3.21}
\hat\f= \lb 0, {\{m_1+i\,P_2m_0,\,0, \, 0 \}\choose \{-P_1m_0, \;
F_0^*(m_1+i\,P_2m_0), \; 0\}}  \rb \in \gH^2\oplus (\cH_0\oplus \cH_1)
\end{equation}
satisfies \eqref{3.12} and \eqref{3.13} and by assertion (a) in the proof of
Proposition \ref{pr3.3} $\hat\f$ satisfies the identity \eqref{3.7} for all
$\{{f\choose f'}, {h_0\choose h_1}\}\in\G$. Therefore by Definition
\ref{def3.2} and \eqref{3.21} $\hat\f\in \G_\infty$, which in view of
\eqref{3.17} gives $m_0=0$ and $m_1=0$. This implies that $\ran
G=\cH_0\oplus\cH_1$.
 \end{proof}
The following two corollaries arise from Propositions \ref{pr3.3} and
\ref{pr3.4}.
\begin{corollary} \label{cor3.5}
Let $\BR$ be a boundary relation for $A^*$ with $\dim\cH_0 < \infty$ and let
$n_\G:=\dim (\mul\G)$. Then: 1)  $n_-(A)\leq n_+(A)<\infty$ and
\begin {equation}\label{3.22}
\dim \cH_0=n_+(A)+n_\G, \qquad \dim \cH_1=n_-(A)+n_\G.
\end{equation}

2) in the case $\mul\G=\{0\}$ (and only in this case) the relation $\G$ turns
into the boundary triplet for $A^*$. More precisely, if $\mul\G=\{0\}$, then
$\G_0=P_{\cH_0\oplus \{0\}}\G$ and $\G_1=P_{\{0\}\oplus\cH_1 }\G$ are operators
and $\Pi=\bta$ is a boundary triplet for $A^*$.
\begin{proof}
1) Let $\{\cK_0\oplus\cK_1, G_0,G_1\}$ be a boundary triplet for $A^*$ defined
in Proposition \ref{pr3.4}, 5). Then according to \eqref{2.4} one has $\dim
\cK_0=n_+(A)$ and $\dim \cK_1=n_-(A)$. Moreover, \eqref{3.17} implies that
$\dim (\cK_\G'\oplus\cK_\G'')=n_\G$. This and decompositions \eqref{3.14.16} of
$\cH_0$ and $\cH_1$ yield \eqref{3.22}.

2) If $\mul\G=\{0\}$, then by \eqref{3.17} $\cK_\G'=\cK_\G''=\{0\}$. Therefore
in view of \eqref{3.14.16} and \eqref{3.10} the boundary triplet
$\{\cK_0\oplus\cK_1, G_0,G_1\}$ satisfies the equalities $\cK_j=\cH_j$ and
$G_j=\G_j, \; j\in\{0,1\}$.
\end{proof}
\end{corollary}
\begin{corollary} \label{cor3.6}
Assume that $A$ is a closed symmetric linear relation in $\gH$, $\cH_0$ is a
Hilbert space with $\dim\cH_0 < \infty$, $\cH_1$ is a subspace in $\cH_0$ and
$\BR$ is a linear relation such that $\dom\G=A^*$ and $\ker\G=A$. If the
identity \eqref{3.7} is satisfied for $\G$ and
\begin {equation}\label{3.23}
\dim (\HH)=n_+(A)+n_-(A)+2\,n_\G,
\end{equation}
then $\G$ is a boundary relation for $A^*$.
\end{corollary}
\begin{proof}
Applying the same arguments as in the proof of Proposition \ref{pr3.4} one
obtains  decompositions \eqref{3.14.16} of $\cH_0$ and $\cH_1$ and the equality
\eqref{3.10} with $\cK'=\cK_\G', \; \cK''=\cK_\G''$ and operators $G_j:A^*\to
\cK_j, \; j\in\{0,1\},$ satisfying the Green's identity \eqref{2.3}. Moreover,
by \eqref{3.17} $\dim (\cK_\G' \oplus\cK_\G'')=n_\G$ and \eqref{3.14.16}
together with \eqref{3.23} gives $\dim (\cK_0\oplus\cK_1)=n_+(A)+n_-(A)$.
Observe also that in view of \eqref{3.10} $\ker G=\ker \G= A $ (here
$G=(G_0\;\;G_1)^\top$) and hence
\begin {equation*}
\dim (\dom G / \ker G)=\dim (A^* / A) =n_+(A)+n_-(A)=\dim (\cK_0\oplus\cK_1).
\end{equation*}
This implies that $\ran G=\cK_0\oplus\cK_1$ and, consequently, the operators
$G_0$ and $G_1$ form the boundary triplet $\{\cK_0\oplus\cK_1, G_0,G_1\}$ for
$A^*$.  Therefor by Proposition \ref{pr3.3} $\G$ is the boundary relation for
$A^*$.
\end{proof}
In the case $\cH_0=\cH_1=:\cH$ the above statements on boundary relations can
be rather simplified. Namely, the following corollary is immediate from
Propositions \ref{pr3.3} - \ref{pr3.4}.
\begin{corollary} \label{cor3.7}

Assume that $\Pi=\{\cK,G_0,G_1\}$ is a boundary triplet for $A^*$ (in the sense
of \cite{GorGor}), $\cK'$ and $\cK''$ are Hilbert spaces and
\begin {equation*}
\cH:=\cK\oplus \cK'\oplus\cK''.
\end{equation*}
Moreover, let $F\in [\cK',\cK]$ and $ F'=(F')^*\in [\cK']$ be linear operators.
Then the equality
\begin {equation}\label{3.25}
\G=\lb \lb \hat f, {\{G_0\hat f,\, k',\, 0\}\}\choose \{G_1\hat f +F k',\, F^*
G_0\hat f +F'k',\, k''\}} \rb :\hat f\in A^*,\; k'\in\cK', \; k''\in\cK'' \rb
\end{equation}
defines the boundary relation $\G:\gH^2\to\cH^2$ for $A^*$.

If in addition $\cK''=\{0\}$, then the following statements are valid:

1) the equality \eqref{3.14.3} defines the self-adjoint extension $A_0$ of $A$
and $A_0=\ker G_0$;

2) the relations
\begin{gather}
\hat\g (\l)= (\G_0\up\hat \gN_\l (A))^{-1},\qquad \g (\l)=\pi_1\hat\g
(\l),\;\; \;\;\l\in\rho (A_0);\label{3.26}\\
 {\rm gr} M(\l)=\lb \{h,h'\}\in \cH^2: \lb {f\choose \l f}, {h\choose h'}\rb
\in\G \;\;\text{for some} \;\; f\in\gH \rb, \;\;\l\in\rho (A_0), \label{3.27}
\end{gather}
define the $\g$-field $\g(\cdot):\rho (A_0) \to[\cH,\gH]$ and the Weyl function
$M(\cdot):\rho (A_0)\to [\cH]$ corresponding to $\G$. Moreover, $\g(\l)$ and
$M(\l)$ are associated with the $\g$-field $\g_\Pi(\l)$ and the Weyl function
$M_\Pi(\l)$ for the triplet $\Pi $ (see \eqref{2.10}) via
\begin {equation}\label{3.28}
\g(\l)=\g_{\Pi}(\l)P_{\cK}, \qquad M(\l)=\begin{pmatrix} M_{\Pi}(\l) & F \cr
F^* & F'\end{pmatrix}: \cK\oplus\cK'\to \cK\oplus\cK', \quad \l\in\rho (A_0)
\end{equation}
and the following identities hold
\begin{gather}
\g(\mu)=\g(\l)+(\mu-\l)(A_0-\mu)^{-1} \g(\l), \;\;\; \mu,\l\in\rho
(A_0)\label{3.29}\\
M(\mu)-M^*(\l)=(\mu-\ov\l) \g^*(\l)\g(\mu), \quad \mu,\l\in\rho
(A_0).\label{3.30}
\end{gather}
The identity \eqref{3.30} implies that $M(\cd)$ is a Nevanlinna operator
function.

Conversely, let $\G:\gH^2\to\cH^2$ be a boundary relation for $A^*$ with
$\dim\cH<\infty$  and let $\cK_\G'=\dom (\mul\G), \; \cK_\G''=\mul (\mul\G)$
(c.f. \eqref{3.3} and \eqref{3.4}). Then
\begin {equation*}
n_+(A)=n_-(A)=\dim \cH-n_\G
\end{equation*}
(here $n_\G=\dim (\mul\G)$) and the following statements hold:

1) $\cH=\cK\oplus\cK_\G'\oplus \cK_\G''$, where $\cK=\cH\ominus (\cK_\G'\oplus
\cK_\G'')$;

2) $\G$ admits the representation  \eqref{3.25} with some boundary triplet
$\Pi=\{\cK,G_0,G_1\}$ for $A^*$ and operators  $F\in [\cK_\G',\cK]$ and $
F'=(F')^*\in [\cK_\G']$.
\end{corollary}
\begin{remark}\label{rem3.8}
1) In the case $\cH_0=\cH_1=:\cH$ our Definitions \ref{def3.2} and
\ref{def3.2a} of the boundary relation $\G:\gH^2\to\cH^2$ for $A^*$ and the
corresponding Weyl family $M(\cd)$ coincide with that introduced in
\cite{DM06}.

2) The identities \eqref{3.29} and \eqref{3.30} were proved in \cite{DM06} (see
also \cite{DM09}).
\end{remark}
\section{Canonical systems}
\subsection{Notations}
Let $\cI=\langle a,b\rangle\; (-\infty\leq a< b\leq\infty)$ be an interval of
the real line (in the case $a> -\infty $ (resp. $b<\infty$) the endpoint $a$
(resp. $b$) may or may not belong to $\cI$) and let $\bH$ be a
finite-dimensional   Hilbert space. Denote by $\cL_{loc}^1(\cI)$ the set of all
Borel operator functions $F(\cd)$ defined almost everywhere on $\cI$ with
values in $[\bH]$ and such that $\int\limits_{[\a,\b]}||F(t)||\,dt<\infty$ for
any finite segment $[\a,\b]\subset \cI$.

Next, denote by $\AC$ the set of all functions $f(\cd):\cI\to \bH$, which are
absolutely continuous on any segment $[\a,\b]\subset \cI$. Moreover, let $\ACf$
be the set of all functions $f\in\AC$ with the following property: if $a\in
\cI$ (resp. $b\in\cI$), then $f(a)=0$ (resp. $f(b)=0$); otherwise $f(t)=0$ on
some interval $(a,\a)\subset \cI$ (resp. $(\b,b)\subset \cI$). Clearly, in the
case of a finite segment $\cI=[a,b]$ the set $AC_0(\cI)$ coincides with the set
of all functions $f\in \AC$ such that $f(a)=f(b)=0$.

Let $\D(\cd)\in \cL_{loc}^1(\cI)$ be an operator function such that $\D(t)\geq
0 $ a.e. on $\cI$. Denote by $\lI$ the linear space of all Borel functions
$f(t)$ defined almost everywhere on $\cI$ with values in $\bH$ and such that
$\int\limits_{\cI}(\D (t)f(t),f(t))_\bH \,dt<\infty$. As is known \cite{Kac50,
DunSch} $\lI$ is a semi-Hilbert space with the semi-definite inner product
$(\cd,\cd)_\D$ and semi-norm $||\cd||_\D$ given by
\begin {equation*}
(f,g)_\D=\int_{\cI}(\D (t)f(t),g(t))_\bH \,dt, \quad
||f||_\D=((f,f)_\D)^{\frac1 2 }, \qquad f,g\in \lI.
\end{equation*}
The semi-Hilbert space $\lI$ gives rise to the Hilbert space $\LI=\lI /
\{f\in\lI: ||f||_\D=0\}$, i.e., $\LI$ is the Hilbert space of all equivalence
classes ($f$ equivalent to $g$ means $\D(t) (f(t)-g(t))=0$ a.e. on $\cI$) in
$\lI$. The inner product and norm in $\LI$ are defined by
\begin {equation*}
(\wt f, \wt g)=(f,g)_\D, \quad ||\wt f||=(\wt f, \wt f)^{\frac 1 2}=||f||_\D,
\qquad \wt f, \wt g\in\LI,
\end{equation*}
where $f\in\wt f \; (g\in\wt g)$ is any representative of the class $\wt f$
(resp. $\wt g$).

In the sequel we systematically use the quotient map $\pi$ from $\lI$ onto
$\LI$ given by $\pi f=\wt f(\ni f), \; f\in \lI$. Moreover, we let
$\wt\pi:=\pi\oplus\pi: (\lI)^2 \to (\LI)^2$, so that $\wt \pi\{f,g\}=\{\wt f,
\wt g\}, \;\; f,g \in \lI$. It is clear that $\ker \pi=\{f\in\lI:\, \D(t)f(t)=0
\;\;\text{a.e. on}\;\;\cI\}$.
\subsection{Minimal and maximal relations}
Let as above $\cI=\langle a,b\rangle \;(-\infty \leq a <b\leq\infty )$ be an
interval and let $\bH$ be a Hilbert space with $n:=\dim \bH\leq\infty$.
Moreover, let $B(\cd), \D (\cd)\in \cL_{loc}^1(\cI)$ be operator functions such
that $B(t)=B^*(t)$ and $\D(t)\geq 0$ a.e. on $\cI$ and let $J\in [\bH]$ be a
signature operator ( this means that  $J^*=J^{-1}=-J$).

A canonical system (on an interval $\cI$) is a system of differential equations
of the form
\begin {equation}\label{4.1}
J y'(t)-B(t)y(t)=\D(t) f(t), \quad t\in\cI,
\end{equation}
where $f(\cd)\in \lI$. Together with \eqref{4.1} we consider also the
homogeneous canonical system
\begin {equation}\label{4.1.1}
J y'(t)-B(t)y(t)=\l \D(t) y(t), \quad t\in\cI, \quad \l\in\bC.
\end{equation}
A function $y\in\AC$ is a solution of \eqref{4.1} (resp. \eqref{4.1.1}) if the
equality \eqref{4.1} (resp. \eqref{4.1.1} holds a.e. on $\cI$. In the sequel we
denote by $\cN_\l$ the linear space of all solutions of the homogeneous system
\eqref{4.1.1} belonging to $\lI$:
\begin {equation}\label{4.1.2}
\cN_\l=\{y\in \AC\cap \lI: J y'(t)-B(t)y(t)=\l \D(t) y(t) \;\;\text{a.e. on}
\;\; \cI\}, \quad\l\in\bC.
\end{equation}
It follows from \eqref{4.1.2} that $\dim \cN_\l\leq \dim \bH <\infty$.

As was shown in \cite{KogRof75} the set of all solutions of \eqref{4.1.1} such
that $\D(t)y(t)=0$ (a.e. on $\cI$) does not depend on $\l$. This enables one to
introduce the following definition.
\begin{definition}\label{def4.0}$\,$\cite{KogRof75}
The null manifold $\cN$ of the system \eqref{4.1} is the subspace of
$\cN_\l\;(\l\in\bC)$ given by
\begin {equation}\label{4.1.3}
\cN=\cN_\l\cap\ker\pi=\{y\in \AC: J y'(t)-B(t)y(t)=\l \D(t) y(t)\;\; \text{and}
\;\; \D(t)y(t)=0 \;\;\text{a.e. on} \;\; \cI\}.
\end{equation}
\end{definition}
For each $c\in\cI$ denote  by $\bH_c$ the subspace
\begin {equation}\label{4.1.4}
\bH_c=\{y(c): y(\cd)\in \cN\}(\subset \bH)
\end{equation}
and let
\begin {equation}\label{4.1.5}
k_\cN=\dim \cN=\dim \bH_c.
\end{equation}
Clearly, $\cN\subset\cN_\l \; (\l\in\bC)$ and for any fixed $\l_0 \in \bC
\setminus \bR$
\begin {equation}\label{4.1.6}
\cN=\cN_{\l_0}\cap \cN_{\ov \l_0}.
\end{equation}
According to \cite{Orc,Kac83,Kac84,LesMal03} the canonical system \eqref{4.1}
induces the \emph{maximal relations} $\tma$ in $\lI$ and $\Tma$ in $\LI$, which
are defined by
\begin {gather}
\tma=\{\{y,f\}\in\lI\times\lI:y\in\AC \;\;\text{and}\;\; J y'(t)-B(t)y(t)=\D(t)
f(t)\;\;\text{a.e. on}\;\; \cI \},\label{4.2}\\
\Tma=\wt\pi \tma=\{\{\wt y, \wt f\}\in\LI\oplus\LI: \wt y=\pi
y\;\;\text{and}\;\; \wt f=\pi f\;\;\text{for some}\;\;
\{y,f\}\in\tma\}.\label{4.3}
\end{gather}
For $\{y,f\}, \; \{z,g\}\in\tma$ and a segment $[\a,\b]\subset\cI$ integration
by parts yields
\begin {equation*}
\int_{[\a,\b]}(\D(t)f(t),z(t))\,dt-\int_{[\a,\b]}(\D(t)y(t),g(t))\,dt=(J
y(\b),z(\b))-(J y(\a),z(\a)).
\end{equation*}
Hence there exist the limits
\begin {equation}\label{4.4}
[y,z]_a:=\lim_{\a \downarrow a}(J y(\a),z(\a)),\qquad [y,z]_b:=\lim_{\b
\uparrow b}(J y(\b),z(\b)), \quad y,z \in\dom\tma
\end{equation}
and the following Lagrange's identity holds
\begin {equation}\label{4.5}
(f,z)_\D-(y,g)_\D=[y,z]_b -[y,z]_a,\quad \{y,f\}, \; \{z,g\} \in\tma.
\end{equation}
Formula \eqref{4.4} defines the boundary bilinear forms $[\cd,\cd]_a $ and
$[\cd,\cd]_b $ on $\dom \tma$, which play an essential role in what follows. By
using this form we define the \emph{minimal relations} $\tmi$ in $\lI$ and
$\Tmi$ in $\LI$ via
\begin {gather}
\tmi=\{\{y,f\}\in\tma:[y,z]_a=0 \;\; \text{and}\;\; [y,z]_b=0\;\;\text{for
every}\;\; z\in \dom \tma \},\label{4.6}\\
\Tmi= \wt\pi \tmi=\{\{\wt y, \wt f\}\in\LI\oplus\LI: \,\wt y=\pi
y\;\;\text{and}\;\; \wt f=\pi f\;\;\text{for some}\;\;
\{y,f\}\in\tmi\}.\label{4.7}
\end{gather}
Moreover, introduce linear relations $\cT_0$ in $\lI$ and $T_0$ in $\LI$ by
letting
\begin {gather}\label{4.8}
\cT_0=\{\{y,f\}\in \tma:\,  y\in\ACf\}, \quad T_0=\wt\pi\cT_0.
\end{gather}
It is clear that $\cT_0\subset \tmi\subset\tma$ and
$T_0\subset\Tmi\subset\Tma$.

Our next goal is to show that $\Tmi$ is a closed symmetric linear relation and
$T_{min}^*=\Tma$. To do this we start with the following definition.
\begin{definition}\label{def4.1}
A finite endpoint $a$ (resp. $b$) of the interval $\cI=\langle a,b \rangle$ is
said to be a regular endpoint of the canonical system \eqref{4.1} if $a\in\cI$
(resp. $b\in\cI$). The canonical system \eqref{4.1} is called regular if both
endpoints $a$ and $b$ are regular; otherwise it is called singular.
\end{definition}
Clearly, in the case of the regular endpoint $a$ (resp. b) integrals
$\int\limits_{[a,c]}||B(t)||\, dt$ and $\int\limits_{[a,c]}||\D(t)||\, dt$
(resp. $\int\limits_{[c,b]}||B(t)||\, dt$ and $\int\limits_{[c,b]}||\D(t)||\,
dt$) are finite for any $c\in (a,b)$.

If the system \eqref{4.1} is regular, then the identity \eqref{4.5} can be
written as
\begin {equation}\label{4.10}
(f,z)_\D-(y,g)_\D=(J y(b),z(b)) -(J y(a),z(a)),\quad \{y,f\}, \; \{z,g\}
\in\tma.
\end{equation}
In the case of the regular system \eqref{4.1} we associate with every subspace
$\bK\subset\bH$ two pairs of linear relations $\cT_\bK, \;\cT_{\bK *} $ in
$\lI$ and $T_\bK, \; T_ {\bK^*} $ in $\LI$ given by
\begin {gather}
\cT_\bK=\{\{y,f\}\in\tma: y(a)\in\bK \;\;\text{and} \;\; y(b)=0\},\label{4.11}\\
\cT_{\bK *}=\{\{y,f\}\in\tma: y(a)\in (J\bK)^\perp\}, \label{4.12}\\
T_\bK=\wt\pi \cT_\bK, \qquad T_{\bK *}=\wt\pi \cT_{\bK *}.\label{4.13}
\end{gather}
\begin{lemma}\label{lem4.2}
If the system \eqref{4.1} is regular, then for any subspace  $\bK\in\bH$ and
$\l\in\bC$
\begin {equation}\label{4.14}
\ran (T_\bK-\l)=(\ker (T_{\bK*}-\ov\l))^\perp.
\end{equation}
\end{lemma}
\begin{proof}
It follows from \eqref{4.11} and \eqref{4.13} that $\ran (T_\bK-\l)$ is the set
of all $\wt f\in\LI$ with the following property: there are $f\in\wt f$ and
$y\in\AC$ such that
\begin {equation}\label{4.15}
y(a)\in\bK, \quad y(b)=0 \;\;\text{and}\;\; \{y,f+\l y\}\in\tma.
\end{equation}
On the other hand, \eqref{4.12} and \eqref{4.13} imply that
\begin {equation}\label{4.16}
\ker (T_{\bK*}-\ov\l)=\{\wt z\in\LI:\, \{z, \ov\l z\}\in\tma \;\;\text{and}\;\;
z(a)\in (J\bK)^\perp \;\;\text{for some}\;\; z\in\wt z \}.
\end{equation}
Let $\wt f\in \ran (T_\bK-\l), \; \wt z\in \ker (T_{\bK*}-\ov\l)$ and let
$\{y,f+\l y\}, \; \{z, \ov\l z\}$ be the corresponding elements of $\tma$ from
\eqref{4.15}, \eqref{4.16}. Applying to these elements the Lagrange's identity
\eqref{4.10} one obtains
\begin {equation*}
(f,z)_\D=-(Jy(a), z(a))=0.
\end{equation*}
Hence $(\wt f,\wt z)=0$ and, consequently, $\ran (T_\bK-\l)\subset (\ker
(T_{\bK*} -\ov\l))^\perp$.

To prove the inverse inclusion assume that $\wt f\in (\ker (T_{\bK*} -\ov\l))
^\perp$ and let $f\in\wt f, \; f\in\lI$. Moreover, let $y\in\AC$ be the
solution of the equation
\begin {equation*}
Jy'-B(t)y=\D(t) (f(t)+\l y)
\end{equation*}
with the initial data $y(b)=0$, so that $\{y, f+\l y\}\in\tma$. Next, for every
$h\in (J\bK)^\perp$ let $z_h\in \AC$ be the solution of the equation
\begin {equation*}
Jz'-B(t)z= \ov\l \D(t)z
\end{equation*}
with the initial data $z_h(a)=h$ and let $\wt z_h=\pi z_h$. Since $\{z_h, \ov\l
z_h \}\in \tma$ and $z_h(a)\in (J\bK)^\perp $, it follows from \eqref{4.16}
that $\wt z_h\in \ker (T_{\bK*}-\ov\l)$ and, therefore, $(\wt f, \wt z_h)=0$.
Now application of the identity \eqref{4.10} to $\{y, f+\l y\}$ and $\{z_h,
\ov\l z_h\}$ gives
\begin {equation*}
(Jy(a), h)=(Jy(a),z_h(a))-(Jy(b),z_h(b))=-(f,z_h)_\D=-(\wt f, \wt z_h)=0,
\;\;\;\; h\in (J\bK)^\perp,
\end{equation*}
which implies that $y(a)\in\bK$. Thus for an arbitarry $\wt f\in (\ker
(T_{\bK*} -\ov\l)) ^\perp$ we have constructed $f\in\wt f$ and $y\in\AC$
satisfying \eqref{4.15}. This gives the requiered inclusion $(\ker (T_{\bK*}
-\ov\l)) ^\perp \subset \ran (T_\bK-\l)$.
\end{proof}

\begin{lemma}\label{lem4.3}
If the system \eqref{4.1} is regular, then for any subspace  $\bK\subset\bH$
\begin {equation}\label{4.17}
(T_\bK)^*=T_{\bK *}.
\end{equation}
In the particular case $\bK=\{0\}$ formula \eqref{4.17} gives $T_0^*=\Tma$.
\end{lemma}
\begin{proof}
Applying \eqref{4.10} to $\{y,f\}\in \cT_\bK$ and $\{z,g\}\in \cT_{\bK *}$ we
obtain $ (f,z)_\D-(y,g)_\D=0$. Therefore by \eqref{4.13} one has $T_{\bK
*}\subset (T_\bK)^*$.

Let us prove the inverse inclusion. First observe that in view of \eqref{4.16}
(with $\l=0$) $\dim \ker T_{\bK *}\leq \dim \cN_0 <\infty$. Hence $\ker T_{\bK
*}$ is a closed subspace in $\LI$ and \eqref{4.14} gives
\begin {equation}\label{4.18}
\ker T_{\bK *}=(\ran T_\bK)^\perp.
\end{equation}

Let $\{\wt z, \wt f\}\in (T_\bK)^*$. Choose $f\in\wt f, \; f\in\lI$ and let
$y\in \AC$ by the solution of \eqref{4.1} with $y(a)=0$. Then $\{y,f\}\in
\cT_{\bK *}$ and by \eqref{4.13} $\{\wt y, \wt f\}\in T_{\bK *}(\subset
(T_\bK)^*)$ (here $\wt y=\pi y$). Thus $\{\wt z-\wt y,0\}\in (T_\bK)^*$, which
 implies that $\wt z-\wt y \in (\ran T_\bK)^\perp$.
Therefore by \eqref{4.18} $\wt z -\wt y\in \ker T_{\bK *}$, so that $\{\wt z
-\wt y, 0\}\in T_{\bK *}$. Now representing $\{\wt z, \wt f\}$ as
\begin {equation*}
\{\wt z, \wt f\}=\{\wt z -\wt y, 0\}+\{\wt y, \wt f\}
\end{equation*}
and taking into account that both terms in the right hand part belong to
$T_{\bK *}$ one obtains $\{\wt z, \wt f\}\in T_{\bK *}$. This proves the
desired inclusion $(T_\bK)^*\subset T_{\bK *}$.
\end{proof}
\begin{lemma}\label{lem4.4}
Let the singular canonical system \eqref{4.1} be defined  on an interval
$\cI=\langle a,b \rangle$. For every finite segment $\cI'=[a',b']\subset\cI$
denote by $\cT_{max}^{\cI'}$ and $T_{max}^{\cI'}$ maximal relations in
$\cL_\D^2(\cI')$ and $L_\D^2(\cI')$ respectively induced by the restriction of
the system \eqref{4.1} onto $\cI'$. Then there exist a finite segment
$\cI_0'\subset \cI $, a point $c\in\cI_0'$ and a subspace $\bH_0\subset \bH$
with the following property: for any segment $\cI'\supset\cI_0'$ and for any
$\{\wt y, \wt f\}\in T_{max}^{\cI'}$ there exists a unique function
$\overset\circ y \in AC (\cI') $ such that $\overset\circ y\in \wt y,\;\;
\overset\circ y(c)\in \bH_0 $ and $\{\overset\circ y, f\}\in \cT_{max}^{\cI'}$
for any $f\in\wt f$.
\end{lemma}
\begin{proof}
Fix a point $c\in\cI $ and for any segment $\cI'\ni c$ put
\begin {gather}
\cN^{\cI'}=\{y\in AC(\cI'): \, J y'(t)-B(t)y(t)=0 \;\;\text {and}\;\;
\D(t)y(t)=0\;\;\text{a.e. on}\;\;\cI'\},\label{4.19}\\
\bH^{\cI'}=\{y(c):\, y(\cd)\in \cN^{\cI'}\}.\nonumber
\end{gather}
Clearly, the inclusion $\cI_1'\subset\cI_2'$ yields $\bH^{\cI_2'}\subset
\bH^{\cI_1'}$. Since $\dim \bH<\infty$, this implies that there exists a finite
segment $\cI_0'\subset\cI$ such that $\bH^{\cI'}=\bH^{\cI_0'}$ for all
$\cI'\supset \cI_0'$. Put $\bH_0:=(\bH^{\cI_0'})^\perp$ and show that such
$\cI_0'$ and  $\bH_0$ have the desired property.

If $\cI'\supset \cI_0'$ and $\{\wt y, \wt f\}\in T_{max}^{\cI'}$, then
according to definition \eqref{4.3} there is a function $y\in AC (\cI')$ such
that $y\in \wt y $ and for any $f\in\wt f$ the equality \eqref{4.1} holds a.e.
on $\cI'$. Let $\overset\circ y \in AC(\cI')$  be the solution of the equation
\eqref{4.1} on $\cI'$ with the initial data $\overset\circ
y(c)=P_{\bH_0}y(c)(\in\bH_0)$ and let $\f= y- \overset\circ y$. Then $J
\f'(t)-B(t)\f(t)=0$ a.e. on $\cI'$ and $\f(c)=y(c)-\overset\circ
y(c)\in\bH_0^\perp$. Since $\bH_0^\perp=\bH^{\cI_0'}=\bH^{\cI'}$, it follows
that $\f(c)\in \bH^{\cI'}$. Hence $\f\in\cN^{\cI'}$ and, consequently,
$\D(t)(y(t)-\overset\circ y(t))=\D(t)\f(t)=0$ a.e. on $\cI'$. This means that
$\overset\circ y\in \wt y, \; \overset\circ y(c)\in \bH_0$ and $\{\overset\circ
y, f\}\in \cT_{max}^{\cI'}$.

To prove uniqueness of such $\overset\circ y$ assume that $z\in AC (\cI')$ has
the same properties, i.e., $z\in \wt y,\;\; z(c)\in \bH_0$ and $\{z, f\}\in
\cT_{max}^{\cI'}$ for any $f\in\wt f$. Then the function $\psi:= \overset \circ
y -z$ satisfies the equalities $J\psi'(t)-B(t)\psi (t)=0 $ and $\D(t)\psi(t)=0$
a.e. on $\cI'$. Hence $\psi\in\cN^{\cI'}$ and, consequently, $\psi
(c)\in\bH^{\cI'}(=\bH_0^\perp)$. On the other hand, $\psi(c)\in\bH_0$, so that
$\psi(c)=0$. Therefore $\psi=0$ and hence $\overset \circ y=z$.
\end{proof}
\begin{proposition}\label{pr4.5}
Let  $T_0$ be the linear relation in $\LI$ given by \eqref{4.8} . Then
\begin {gather}
T_0^*=\Tma\label{4.22}
\end{gather}
\end{proposition}
\begin{proof}
In the case of the regular system \eqref{4.1} the equality \eqref{4.22} was
proved in Lemma \ref{lem4.3}.

Assume now that the system \eqref{4.1} is singular. Then  applying the
Lagrange's identity \eqref{4.5} to $\{y,f\}\in\tma$ and $\{z,g\}\in\cT_0 $ we
obtain
\begin {equation}\label{4.23}
(f,z)_\D-(y,g)_\D=0.
\end{equation}
Therefore by \eqref{4.3} and \eqref{4.8} one has $\Tma\subset T_0^*$.

Let us prove the inverse inclusion $T_0^*\subset \Tma$. Assume that $\{\wt y,
\wt f\} \in T_0^*$ and choose $y,f\in \lI$ such that $\pi y=\wt y$ and $\pi
f=\wt f$. Next, for every finite segment  $\cI'=[a',b']\subset\cI$ denote by
$y_{\cI'}$ and $f_{\cI'}$ the restrictions  of the functions $y(\cd) $ and
$f(\cd)$ onto $\cI'$ and let $\wt y_{\cI'}=\pi_{\cI'}y_{\cI'}, \; \wt
f_{\cI'}=\pi_{\cI'}f_{\cI'}$, where $\pi_{\cI'}$ is the quotient map from
$\cL_\D^2(\cI')$ onto  $L_\D^2(\cI')$. Assume also that  $\cT_0^{\cI'}$ and
$T_0^{\cI'}$ are linear relations \eqref{4.8} in $\cL_\D^2(\cI')$ and
$L_\D^2(\cI')$ respectively induced by the restriction of the system
\eqref{4.1} onto $\cI'$.

Let $\{z_{\cI'}, g_{\cI'}\}\in \cT_0^{\cI'}$ and let $z(t)$ and $g(t)\;
(t\in\cI)$ be functions obtained from $z_{\cI'}$ and $g_{\cI'}$ by means of
their  zero continuation onto $\cI$.  Then $\{z,g\}\in \cT_0$ and,
consequently, \eqref{4.23} holds. Therefore
\begin {equation*}
\int_{\cI'}(\D(t)f_{\cI'}(t), z_{\cI'}(t))\, dt -\int_{\cI'}(\D(t)y_{\cI'}(t),
g_{\cI'}(t))\, dt=0, \quad \{z_{\cI'},g_{\cI'}\}\in\cT_0^{\cI'},
\end{equation*}
which implies that $\{\wt y_{\cI'}, \wt f_{\cI'}\}\in (T_0^{\cI'})^*$.
Moreover, $(T_0^{\cI'})^*=T_{max}^{\cI'}$, because the restriction of
\eqref{4.1} onto $\cI'$ is a regular system. Thus, $\{\wt y_{\cI'}, \wt
f_{\cI'}\} \in T_{max}^{\cI'}$ for every finite segment $\cI'\subset\cI$.

Next, by Lemma \ref{lem4.4} there exist a finite segment $\cI_0'\subset \cI $,
a point $c\in\cI_0'$ and a subspace $\bH_0\subset \bH$ with the following
property: for any finite segment $\cI'\supset\cI_0'$  there exists a unique
function $\overset\circ y_{\cI'} \in AC (\cI') $ such that
$\pi_{\cI'}\overset\circ y_{\cI'}= \wt y_{\cI'},\;\; \overset\circ
y_{\cI'}(c)\in \bH_0 $ and $\{\overset\circ y_{\cI'}, f_{\cI'}\}\in
\cT_{max}^{\cI'}$. Moreover, by using uniqueness of the function $\overset\circ
y_{\cI'}$ (for a given $\cI'$) one can easily verify that for any pair of
segments $\cI_1', \; \cI_2'$ such that $\cI_0'\subset\cI_1'\subset
\cI_2'\subset\cI $ the restriction $\overset\circ y_{\cI_2'}\up \cI_1' $
coincides with $\overset\circ y_{\cI_1'}$. This allows us to introduce the
function $\overset\circ y\in \AC$ by setting $\overset\circ y(t)=\overset\circ
y_{\cI'}(t), \; t\in\cI$, where $\cI'$ is an arbitrary segment such that
$\cI_0'\subset\cI'\subset \cI $ and $t\in\cI'$. It is clear that $\pi
\overset\circ y=\wt y$ and $\{\overset\circ y, f\}\in \tma$, which gives the
inclusion $\{\wt y, \wt f\}\in \Tma$. Therefore  $T_0^*\subset \Tma$ and the
equality \eqref{4.22} is valid.
 \end{proof}
\begin{lemma}\label{lem4.6}
Let the canonical system \eqref{4.1} be given on an interval $\cI=[a,b\rangle$
with the regular endpoint $a$. Assume also that $\cT_1, \;\cT_2$ and $T_1,\;
T_2$ are linear relations in $\lI$ and $\LI$ respectively defined by
\begin{gather}
\cT_1=\{\{y,f\}\in\tma:\, [y,z]_b=0 \;\;\text{for every}\;\; z\in\dom\tma\},
\qquad T_1=\wt\pi\cT_1,\label{4.24}\\
\cT_2=\{\{y,f\}\in\tma:\, y(a)=0\}, \qquad T_2=\wt\pi\cT_2.\label{4.25}.
\end{gather}
Then
\begin{gather}\label{4.26}
T_1^*=T_2.
\end{gather}
\end{lemma}
\begin{proof}
The inclusion $T_2\subset T_1^*$ follows from the identity \eqref{4.5} applied
to $\{y,f\}\in \cT_1$ and $\{z,g\}\in\cT_2$.

To prove the inverse inclusion assume that $\{\wt y, \wt f\}\in T_1^*$ and let
$y,f\in \lI, \; \pi y=\wt y, \; \pi f =\wt f$. Moreover, for every $\b\in\cI$
let $\cI_\b:=[a,\b]$, let $y_\b$ and $f_\b$ be the restrictions of the
functions $y(\cd)$ and $f(\cd)$ onto $\cI_\b$ and let $\wt y_\b=\pi_\b y_\b, \;
\wt f_\b=\pi_\b f_\b $, where $\pi_\b$ is the quotient map from
$\cL_\D^2(\cI_\b)$ onto $L_\D^2 (\cI_\b)$. Consider also linear relations
$\cT^\b, \; \cT^\b_1$ in $\cL_\D^2 (\cI_\b)$ and  $T^\b, \; T^\b_1$ in $L_\D^2
(\cI_\b)$ given by
\begin{gather}
\cT^\b=\{\{y,f\}\in\cT_{max}^\b:\, y(\b)=0\}, \qquad
T^\b=\wt\pi_\b\cT^\b\nonumber\\
\cT_2^\b=\{\{y,f\}\in\cT_{max}^\b:\, y(a)=0\}, \qquad
T_2^\b=\wt\pi_\b\cT_2^\b.\label{4.27}
\end{gather}
It follows from \eqref{4.17} (with $\bK=\bH$) that $(T^\b)^*=T_2^\b$ and the
same arguments as in the proof of Proposition \ref{pr4.5} give the inclusion
$\{\wt y_\b, \wt f_\b\}\in T_2^\b, \;\b\in\cI$.

Next, according to definition \eqref{4.27} of $T_2^\b$ there is a function $\ov
y_\b\in AC (\cI_\b)$ such that $\pi_\b \ov y_\b=\wt y_\b, \; \ov y_\b(a)=0$ and
$J \ov y_\b'(t)-B(t)\ov y_\b(t)=\D(t)f_\b(t) $ a.e. on $\cI_\b$. Moreover, it
is easily seen that for a given $\b\in\cI$ such a function is unique, so that
$\ov y_{\b_1}=\ov y_{\b_2}\up \cI_{\b_1}$ for any $ \b_1<\b_2$.  Therefore the
equality $\ov y(t)=\ov y_\b(t), \; t\in\cI, \; \b>t$ correctly defines the
function $\ov y\in \AC$ such that $\pi \ov y=\wt y$ and $\{\ov y, f\}\in\cT_2$.
This implies that $\{\wt y, \wt f\}\in T_2$ and hence $T_1^*\subset T_2$.
\end{proof}
\begin{proposition}\label{pr4.7}
Let $a$ be a regular endpoint of the canonical system \eqref{4.1} and let
$\cT_a$ and $T_a$ be linear relations in $\lI$ and $\LI$ respectively given by
\begin{gather}\label{4.28}
\cT_a=\{\{y,f\}\in\tma: y(a)=0 \;\;\text{and}\;\;\, [y,z]_b=0 \;\;\text{for
every}\;\; z\in\dom\tma\}, \qquad T_a=\wt\pi\cT_a.
\end{gather}
Moreover, let $T_0$ be the relation \eqref{4.8} and let $\ov T_0$ be the
closure of $T_0$. Then  $T_a$ is a closed symmetric relation  and
\begin{gather}
T_a=\ov T_0\label{4.30}\\
T_a^*=\Tma\label{4.31}
\end{gather}
Similarly if $b$ is a regular endpoint of the system \eqref{4.1} and
\begin{gather}\label{4.32}
\cT_b=\{\{y,f\}\in\tma: y(b)=0 \;\;\text{and}\;\;\, [y,z]_a=0 \;\;\text{for
every}\;\; z\in\dom\tma\}, \qquad T_b=\wt\pi\cT_b,
\end{gather}
then $T_b$ is a closed symmetric relation in $\LI$ and
\begin{gather}\label{4.34}
T_b=\ov T_0, \qquad T_b^*=\Tma.
\end{gather}
\end{proposition}
\begin{proof}
Applying the Lagrange's identity \eqref{4.5} to $\{y,f\}\in \cT_a$ and $\{z,
g\}\in \tma$ one obtains the equality \eqref{4.23}. Therefore
\begin{gather}\label{4.35}
\Tma\subset T_a^* \;\;\;\text{and}\;\;\; T_a\subset \Tma^*.
\end{gather}
Moreover, by \eqref{4.28} $T_a\subset \Tma$, which together with the first
inclusion in \eqref{4.35} shows that $T_a$ is symmetric.

Next, assume that $T_1$ and $T_2$ are the linear relations \eqref{4.24} and
\eqref{4.25}. Since $T_1\subset\Tma$, it follows that $\Tma^*\subset T_1^*$ and
by \eqref{4.26} $\Tma^*\subset T_2$. Therefore for any $\{\wt y, \wt
f\}\in\Tma^*$ there is $\{y,f\}\in \tma$ such that $y(a)=0$ and $\wt \pi
\{y,f\}=\{\wt y, \wt f\}$. This and the identity \eqref{4.5} give
\begin{gather*}
[y,z]_b=(f,z)_\D-(y,g)_\D=0, \qquad z\in\dom \tma,
\end{gather*}
which implies that $\{\wt y, \wt f\}\in T_a$. Hence $\Tma^*\subset T_a$ and by
the second inclusion in \eqref{4.35}
\begin{gather}\label{4.36}
T_a=\Tma^*.
\end{gather}
Therefore $T_a$ is closed. Moreover, by \eqref{4.22} $\Tma$ is also closed,
which together with \eqref{4.36} gives \eqref{4.31}. Finally, combining
\eqref{4.22} with \eqref{4.31} we arrive at \eqref{4.30}.

Similarly one proves the relations  \eqref{4.34}.
\end{proof}
\begin{corollary}\label{cor4.8}
Under the assumptions of Lemma \ref{lem4.6} the equality $T_2^*=T_1$ is valid.
\end{corollary}
\begin{proof}
It follows from \eqref{4.31} that for each $\l\in\bC\setminus \bR$ the defect
subspace of the close symmetric relation  $T_a$ is
\begin{gather*}
\gN_\l(T_a)=\ker (\Tma-\l)=\pi \cN_\l.
\end{gather*}
Therefore $T_a$ has finite deficiency indices and \eqref{4.24} gives
$T_a\subset T_1\subset \Tma$. Consequently, $T_1$ is closed and the required
equality $T_2^*=T_1$ follows from \eqref{4.26}.
\end{proof}
\begin{lemma}\label{lem4.9}
Let the canonical system \eqref{4.1} be defined on an interval $\cI=\langle a,
b \rangle $. For each subinterval $\cI_\b:=[\b, b\rangle\;(\b\in\cI) $ denote
by $\cT_{max}^\b$ and $T_{max}^\b$ maximal  relations induced by the
restriction of the system \eqref{4.1} onto $\cI_\b$ and let $\cT_1^\b$ and
$T_1^\b$ be linear relations in $\cL_\D^2(\cI_\b)$ and $L_\D^2(\cI_\b)$
respectively given by
\begin{gather}\label{4.37}
\cT_1^\b=\{\{y,f\}\in\cT_{max}^\b:\, [y,z]_b=0 \;\;\text{for every}\;\;
z\in\dom\cT_{max}^\b\}, \qquad T_1^\b=\wt\pi_\b\cT_1^\b.
\end{gather}
Then there exists a subinterval $\cI_{\b_0}\subset \cI$, a point $c\subset
\cI_{\b_0}$ and a subspace $\hat\bH\subset\bH$ with the following property: for
any interval $\cI_\b\supset \cI_{\b_0}$ and for any $\{\wt y,\wt f\}\in T_1^\b$
there exists a unique function $\hat y\in AC(\cI_\b)$ such that $\pi_\b\hat
y=\wt y, \; \hat y(c)\in\hat\bH$ and $\{\hat y, f\}\in \cT_1^\b$ for any
$f\in\wt f$.
\end{lemma}
\begin{proof}
Fix a point $c\in\cI$ and for any interval $\cI_\b(=[\b,b\rangle)\ni c$ let
\begin{gather*}
\cN^\b=\ker \cT_1^\b\cap \ker \pi_\b=\{y\in AC(\cI_\b): \, \{y,0\}\in\cT_1^\b
\;\;\text{and}\;\; \D(t)y(t)=0 \;\;\text{a.e. on} \;\;
\cI_\b\},\\
\bH^\b = \{y(c):y\in \cN^\b\}.
\end{gather*}
Let us prove the following assertion:

(a)\hskip 2mm if $\cI_{\b_1}\subset \cI_{\b_2}(\iff \b_2 \leq \b_1)$ and
$y_2\in \cN^{\b_2}$, then $y_1:=y_2\up \cI_{\b_1}\in\cN^{\b_1} $.

Indeed, the inclusion $y\in\cN^\b $ is equivalent to the relations
\begin{gather}
J y'(t)-B(t)y(t)=0 \;\;\text{and}\;\; \D(t)y(t)=0\;\; \text{a.e on}\;\; \cI_\b,
\label{4.40}\\
\lim\limits_{t\uparrow b}(J y(t), z(t))=0, \quad \{z,g\}\in \cT_{max}^\b.
\label{4.41}
\end{gather}
 Since $y_1$ is a restriction of $y_2$ and \eqref{4.40} holds for
$y_2$ on $\cI^{\b_2}$, it follows that \eqref{4.40} is valid for $y_1$ on
$\cI^{\b_1}$. Next, assume that $\{z_1, g_1\}\in \cT_{max}^{\b_1} $ and let
$z(t)$ be the solution of the equation
\begin{gather*}
J z'(t)-B(t)z(t)=0, \quad t\in [\b_2,\b_1]
\end{gather*}
such that $z(\b_1)=z_1(\b_1)$. Then the pair $\{z_2,g_2\}$ with
\begin{gather*}
z_2(t)=\begin{cases} z_1(t), \; t\in \cI_{\b_1}\cr z(t), \; t\in [\b_2,\b_1)
\end{cases}, \qquad g_2(t)=\begin{cases} g_1(t), \; t\in \cI_{\b_1}\cr 0,\;
t\in [\b_2,\b_1)
\end{cases}
\end{gather*}
belongs to $\cT_{max}^{\b_2}$ and, consequently, $\lim\limits_{t\uparrow b}(J
y_2(t), z_2(t))=0$. At the same time $z_1=z_2\up \cI_{\b_1}$, so that
$\lim\limits_{t\uparrow b}(J y_1(t), z_1(t))=0$. Hence \eqref{4.41} holds for
$y_1$, which completes the proof of the assertion (a).

It follows from (a) that $\cI_{\b_1}\subset \cI_{\b_2}$ yields $\bH^{\b_2}
\subset \bH^{\b_1}$. Since $\dim \bH<\infty$, this implies that  there exists a
subinterval $\cI_{\b_0}=[\b_0, b\rangle$ such that $\bH^\b=\bH^{\b_0}$ for all
$\cI_\b\supset\cI_{\b_0}$. Next by using the same arguments as in the proof of
Lemma \ref{lem4.4} one shows that the statement of the lemma holds for the
constructed above interval $\cI_{\b_0}$ and the subspace $\hat \bH=
(\bH^{\b_0})^\perp$.
\end{proof}
\begin{lemma}\label{lem4.10}
Let the canonical system \eqref{4.1} be given on an interval $\cI=\langle a,b
\rangle$. Assume also that $\cT_1, \;\cT_3$ and $T_1,\;T_3$ are linear
relations in $\lI$ and $\LI$ respectively given by \eqref{4.24} and
\begin{gather}\label{4.42}
\cT_3=\{\{y,f\}\in\tma: \, [y,z]_a=0 \;\;\text{for every}\;\; z\in\dom\tma\},
\qquad T_3=\wt\pi\cT_3.
\end{gather}
Then
\begin{gather}\label{4.43}
T_3^*=T_1.
\end{gather}
\end{lemma}
\begin{proof}
We give only the sketch of the proof, because it is similar to that of
Proposition \ref{pr4.5} and Lemma \ref{lem4.6}.

The inclusion $T_1\subset T_3^*$ follows from the Lagrange's identity
\eqref{4.5}. To prove the inverse inclusion assume that $\{\wt y, \wt f\}\in
T_3^*$. For every interval $\cI_\b=[\b,b\rangle$ construct the restrictions
$\wt y_\b,\; \wt f_\b \in L_\D^2 (\cI_\b)$ onto $\cI_\b$ in the same way as in
the proof of Lemma \ref{lem4.6}. Moreover, let $\cT_1^\b$ and $T_1^\b$ be the
relations \eqref{4.37} and let $\cT_2^\b=\{\{y,f\}\in\cT_{max}^\b: \,
y(\b)=0\}, \; T_2^\b=\wt\pi_\b \cT_2^\b$. Then by using the Lagrange's identity
one proves the inclusion $\{\wt y_\b,\; \wt f_\b\} \in (T_2^\b)^*$. At the same
time by Corollary \ref{cor4.8} $(T_2^\b)^*=T_1^\b$, so that $\{\wt y_\b,\; \wt
f_\b\} \in T_1^\b$  for every interval $\cI_\b$. Now by using Lemma
\ref{lem4.9} one obtains the inclusion $\{\wt y,\; \wt f\} \in T_1$. Hence
$T_3^*\subset T_1$, which yields \eqref{4.43}.
\end{proof}
Now we are ready to prove the main theorem of the subsection.
\begin{theorem}\label{th4.11}
Let $\Tma$ and $\Tmi$ be maximal and minimal relations \eqref{4.3} and
\eqref{4.7} induced by the canonical system \eqref{4.1} on the interval
$\cI=\langle a, b\rangle$ and let $T_0$ be the relation \eqref{4.8}.  Then
$\Tmi$ is a closed symmetric linear relation in $\LI$ and
\begin{gather}\label{4.44}
\ov T_0=\Tmi, \qquad \Tmi^*=\Tma.
\end{gather}

If in addition the endpoint $a$ (resp. $b$) is regular and $T_a$ (resp. $T_b$)
is the relation \eqref{4.28} (resp. \eqref{4.32}), than $\Tmi=T_a$ (resp.
$\Tmi=T_b$).

If the system \eqref{4.1} is regular, then $\Tmi=T_0$ and every $\l\in\bC$ is a
regular type point of $\Tmi$, that is $\hat\rho (\Tmi)=\bC$.
\end{theorem}
\begin{proof}
It follows from the Lagrange's identity \eqref{4.5} that $\Tma\subset\Tmi^*$
and $\Tmi\subset\Tma^*$. This and the obvious inclusion $\Tmi\subset\Tma$ show
that $\Tmi$ is symmetric.

Next assume that $T_1$  and $T_3$ are the linear relations \eqref{4.24} and
\eqref{4.42}. Since $T_3\subset \Tma$, it follows that $\Tma^*\subset T_3^*$
and by \eqref{4.43} $\Tma^*\subset T_1$. Now the arguments similar to that in
the proof of Proposition \ref{pr4.7} give the equality $\Tma^*=\Tmi$, which
together with \eqref{4.22} leads to \eqref{4.44}. Moreover, combining
\eqref{4.44} with \eqref{4.30} and \eqref{4.34} we arrive at the required
statement for systems with the regular endpoint $a$ or $b$.

Assume now that the system \eqref{4.1} is regular and show that in this case
\begin{gather}\label{4.45}
\ker (T_0-\l)=\{0\}, \qquad \ov{\ran (T_0-\l)}=\ran (T_0-\l), \qquad \l\in\bC.
\end{gather}
If $\wt y\in \ker (T_0-\l)$, then $\{\wt y, \l\wt y\}\in T_0$ and,
consequently, there is $y\in\AC$ such that $\pi y=\wt y, \; y(a)=y(b)=0$ and
$y$ is a solution of \eqref{4.1.1}. Hence $y=0$, which gives the first equality
in \eqref{4.45}. Moreover, formula \eqref{4.14} (with $\bK=\{0\}$) implies the
second equality in \eqref{4.45}.

Since $T_0$ is symmetric, it follows from  \eqref{4.45} that $T_0$ is closed.
Therefore by \eqref{4.44} $\Tmi=T_0$ and \eqref{4.45} yields the equality
$\hat\rho (\Tmi)=\bC$.
\end{proof}
Let $\cN$ be the null manifold \eqref{4.1.3} of the canonical system
\eqref{4.1}. Then $\{y,0\}\in\tma$ for every $y\in\cN$ and the Lagrange's
identity \eqref{4.5} gives
\begin{gather}\label{4.45.1}
[y,z]_a=[y,z]_b, \qquad y\in\cN, \;z\in \dom \tma.
\end {gather}
This enables us to introduce the subspace $\cN'\subset\cN$ via
\begin{gather}
\cN'=\{y\in\cN:\,[y,z]_a =0, \;\; 
z\in\dom\tma\}= \{y\in\cN:\, [y,z]_b=0,\;\;
 z\in\dom\tma\}.\label{4.45.2}
\end{gather}
Next, the relations $\{y,f\}\in\tma$ and $\wt\pi \{y,f\}=0$ mean that $y\in\AC,
\; f\in\lI$ and
\begin{gather*}
J y'(t)-B(t)y(t)=\D(t)f(t), \quad \D(t)y(t)=0, \quad \D(t)f(t)=0
\;\;\;\;\text{a.e. on  }\;\; \cI.
\end{gather*}
Therefore
\begin{gather}
\ker (\wt\pi\up\tma)=\{\{y,f\}\in\lI\times \lI:\, y\in \cN \:\:\text{and}
 \:\;\D(t)f(t)=0\;\;\text{a.e. on}\;\;\cI\},\label{4.45.6}\\
\ker (\wt\pi\up\tmi)=\{\{y,f\}\in\lI\times \lI:\, y\in \cN' \:\:\text{and}
 \:\;\D(t)f(t)=0\;\;\text{a.e. on}\;\;\cI\}.\label{4.45.7}
\end{gather}
\begin{proposition}\label{pr4.11a}
Let $a$ be a regular endpoint of the system \eqref{4.1}, let $\cT_a$ be the
linear relation \eqref{4.28} and let $\hat\cN'=\{\{y,0\}:\,y\in\cN'\}$. Then
\begin{gather}\label{4.45.4}
\tmi=\cT_a\dotplus \hat\cN',
\end{gather}
which implies that the equality $\tmi=\cT_a$ holds if and only if $\cN'=\{0\}$.
\end{proposition}
\begin{proof}
Since $\cT_a\subset\tmi$ and by Theorem \ref{th4.11} $\wt\pi \tmi=\wt\pi\cT_a
(=\Tmi)$, it follows that
\begin{gather}\label{4.45.5}
\tmi=\cT_a+\ker (\wt\pi\up\tmi).
\end{gather}
Clearly, the inclusion $\{0,f\}\in\cT_a$ holds for any $f\in\lI$ with
$\D(t)f(t)=0\;\;\text{a.e. on}\;\;\cI$. Combining this assertion  with
\eqref{4.45.5} and \eqref{4.45.7} one obtains $\tmi=\cT_a+\hat\cN'$. Moreover,
for each $y\in\cN\cap\dom \cT_a$ one has $y(0)=0$ , so that  $y=0$. Hence
$\cT_a\cap\hat\cN'=\{0\}$, which gives the direct decomposition \eqref{4.45.4}.
\end{proof}
\begin{example}\label{ex4.11b}
Consider the canonical system \eqref{4.1} with $\bH=\bC^2$ and operator
coefficients $J, \; B(t)$  and $\D(t)$ given in the standard basis of $\bC^2$
by the matrices
\begin{gather*}
J=\begin{pmatrix} 0 & -1 \cr 1& 0 \end{pmatrix}, \quad B(t)=0, \quad
\D(t)=\begin{pmatrix} 1 & 0 \cr 0 & 0 \end{pmatrix}, \quad t\in [0,\infty).
\end{gather*}
One immediately checks that for this system $\cN_\l=\cN=\{y(t)\equiv \{0,
C\}:\, C\in\bC\}$ and each function $z\in\dom \tma$ is of the form $z(t)=\{0,
\;z_2(t) \}(\in\bC^2)$. Hence $(J y(t), z(t))\equiv 0\;\;(y\in\cN, \;
z\in\dom\tma)$, so that $\cN'=\cN\neq \{0\}$. This example shows that there
exist canonical systems with the regular endpoint $a$ such that $\tmi\neq
\cT_a$.
\end{example}
\subsection{Deficiency indices and Neumann formulas}
Let $\tma$ be the maximal relation \eqref{4.2} in $\lI$ induced by the
canonical system \eqref{4.1} and let $\cN_\l$ be the subspace \eqref{4.1.2}. It
follows from \eqref{4.2} that
\begin{gather}\label{4.46}
\cN_\l=\ker (\tma-\l)=\{y\in\lI:\; \{y,\l y\}\in\tma\}, \quad\l\in\bC.
\end{gather}
Assume also that $\hat\cN_\l$ is a subspace in $\tma$ given by
$\hat\cN_\l=\{\{y,\l y\}: y\in\cN_\l\}, \;\l\in\bC$.
\begin{definition}\label{def4.12} $\,$\cite{KogRof75}
The numbers $N_+=\dim \cN_i$ and $N_-=\dim\cN_{-i} $ are called the formal
deficiency indices of the system \eqref{4.1}.
\end{definition}
It is clear that $N_\pm\leq n$. Moreover, if the system \eqref{4.1} is regular,
then $N_+=N_-=n$.

Next assume that
\begin{gather*}
\gN_\l:=\gN_\l(\Tmi)=\ker (\Tma-\l), \quad \l\in\bC
\end{gather*}
is the defect subspace and
\begin{gather*}
n_\pm:=n_\pm (\Tmi )=\dim \gN_\l, \quad \l\in\bC_\pm
\end{gather*}
are deficiency indices of the symmetric relation $\Tmi$ in $\LI$. It is easily
seen that $\pi\cN_\l=\gN_\l$ and $\ker(\pi\up\cN_\l)=\cN$ for each $\l\in\bC$.
This implies the following proposition.
\begin{proposition}\label{pr4.13}$\,$ \cite{KogRof75,LesMal03}
Given a canonical system \eqref{4.1}. Then $N_\pm=\dim \cN_\l, \; \l\in\bC_\pm$
(i.e., $\dim \cN_\l$ does not depend on $\l$ in either $\bC_+$ or $\bC_-$) and
\begin{gather}\label{4.49}
N_+=n_+ + k_\cN, \quad N_-=n_- + k_\cN.
\end{gather}
\end{proposition}
 As is known (see for instance \cite{Ben72}), for any closed symmetric relation
$A$ in $\gH$ the Neumann formula is valid. In the case of the minimal relation
$\Tmi$ in $\LI$ this formula is
\begin{gather}\label{4.50}
\Tma=\Tmi \dotplus \hat\gN_\l(\Tmi) \dotplus \hat\gN_{\ov\l}(\Tmi),
\quad\l\in\bC\setminus\bR.
\end{gather}
In the following proposition we show that similar formulas hold for $\tmi$ and
$\tma$.
\begin{proposition}\label{pr4.14}
Let $\tmi$ and $\tma$ be minimal and maximal relations in $\lI$ induced by the
system \eqref{4.1}. Assume also that $\cN$ is the null manifold \eqref{4.1.3},
$\cN'\subset \cN$ is the subspace \eqref{4.45.2} and let $k_{\cN'}=\dim \cN'$.
Then:  1) for each $\l\in\bC\setminus\bR$ the following Neumann formulas hold
\begin{gather}\label{4.51}
\tma=\tmi + (\hat\cN_\l \dotplus \hat\cN_{\ov\l}), \qquad \tmi\cap (\hat\cN_\l
\dotplus \hat\cN_{\ov\l})=\cN'\oplus \cN = \{\{y,f\}:\,y\in\cN', \; f\in\cN\}.
\end{gather}

2) the following equality is valid
\begin{gather}\label{4.52}
\dim (\tma/\tmi)=\dim (\dom\tma/\dom\tmi)= N_+ +N_- -k_\cN-k_{\cN'}.
\end{gather}
\begin{proof}
Since $\wt\pi\tma=\Tma, \; \wt\pi\tmi=\Tmi$ and $\wt\pi \hat\cN_\l=\hat\gN_
\l(\Tmi)$, it follows from \eqref{4.50} that
\begin{gather}\label{4.53}
\tma=\tmi + (\hat\cN_\l \dotplus \hat\cN_{\ov\l})+\ker (\wt\pi\up \tma),\quad
\l\in \bC\setminus\bR
\end{gather}
(clearly $\hat\cN_\l\cap \hat\cN_{\ov\l}=\{0\}$, so that the sum $\hat\cN_\l
\dotplus \hat\cN_{\ov\l}$ in \eqref{4.53} is direct). Next, the inclusion
$y\in\cN$ implies that $y\in \cN_\l\cap\cN_{\ov\l}$. Moreover,
\begin{gather*}
\{y,0\}=\left\{
 -\tfrac {\ov\l}{\l-\ov\l} y, -\l\tfrac { \ov\l }{\l-\ov\l} y
\right\} + \left\{ \tfrac {\l}{\l-\ov\l} y, \ov\l\tfrac { \l }{\l-\ov\l} y
\right\}.
\end{gather*}
and consequently
\begin{gather}\label{4.55}
\{y,0\}\in\hat\cN_\l \dotplus \hat\cN_{\ov\l} \;\;\;\;\; (y\in\cN,\;\;
\l\in\CR).
\end{gather}
Furthermore  for each $f\in\lI$ such that $\D(t)f(t)=0$ a.e on $\cI$ one has
$\{0,f\}\in\tmi$. This and \eqref{4.45.6} give the inclusion $\ker
(\wt\pi\up\tma)\subset \tmi + (\hat\cN_\l \dotplus \hat\cN_{\ov\l})$, which
together with \eqref{4.53} yields the first equality in \eqref{4.51}.

Let us prove the second relation in \eqref{4.51}. If $\{y,f\}\in \tmi\cap
(\hat\cN_\l \dotplus \hat\cN_{\ov\l})$, then $\wt\pi \{y,f\}\in \Tmi\cap
(\hat\gN_\l  \dotplus \hat\gN_{\ov\l})$ and hence $\wt\pi \{y,f\}=0$. Therefore
by \eqref{4.45.7} $y\in \cN'$ and in view of \eqref{4.55} $\{y,0\}\in\hat\cN_\l
\dotplus \hat\cN_{\ov\l} $. Moreover, since $\{y,f\}\in\hat\cN_\l \dotplus
\hat\cN_{\ov\l}$ as well, one has $\{0,f\}\in\hat\cN_\l \dotplus
\hat\cN_{\ov\l} $. This implies that there exist $y\in\cN_\l$ and
$z\in\cN_{\ov\l}$ such that $y+z=0$ and $\l y+\ov\l z =f$. Hence
$f=(\ov\l-\l)z=(\l-\ov\l)y$, so that $f\in \cN_\l \cap\cN_{\ov\l}$ and by
\eqref{4.1.6} $f\in\cN$. Thus $\{y,f\}\in \cN'\oplus \cN$.

Conversely, let $\{y,f\}\in \cN'\oplus \cN$  with $y\in\cN'$ and $f\in\cN$.
Then according to \eqref{4.45.2} $\{y,0\}\in\tmi$ and \eqref{4.55} gives
$\{y,0\}\in\hat\cN_\l \dotplus \hat\cN_{\ov\l}$. Therefore the inclusion
$\{y,0\}\in \tmi\cap (\hat\cN_\l \dotplus \hat\cN_{\ov\l})$ is valid. Next,
$\{0,f\}\in \tmi$ and the representation
\begin{gather*}
\{0,f\}=\tfrac 1 {\l-\ov\l}(\{f,\l f\}-\{f,\ov\l f\})
\end{gather*}
together with \eqref{4.1.6} shows that $\{0,f\}\in \hat\cN_\l \dotplus
\hat\cN_{\ov\l}$. Thus $\{0,f\}\in \tmi\cap ( \hat\cN_\l \dotplus
\hat\cN_{\ov\l})$ and therefore $\{y,f\}\in \tmi\cap ( \hat\cN_\l \dotplus
\hat\cN_{\ov\l})$ as well. This proves the second relation in \eqref{4.51}.

To prove \eqref{4.52} we first note that the equality $r:=\dim (\tma /
\tmi)=N_+ +N_- - k_\cN-k_{\cN'}$ is immediate from \eqref{4.51}. Next assume
that $\{\{y_j,f_j\}\}_1^r$ is a basis  of $\tma$ modulo $\tmi$. Then the
immediate checking shows that $\{y_j\}_1^r$ is the basis of $\dom \tma$ modulo
$\dom \tmi $. Therefore $\dim (\dom \tma / \dom \tmi)=r(=\dim (\tma / \tmi))$
which completely proves \eqref{4.52}.
\end{proof}
\end{proposition}
In the following proposition we give a somewhat different form of the Neumann
formulas, which hold in the case of the regular endpoint.
\begin{proposition}\label{pr4.15}
Let $\cT_a$ be the linear relation \eqref{4.28} induced by the system
\eqref{4.1} with the regular endpoint $a$. Then
\begin{gather}\label{4.57}
\tma=\cT_a + (\hat\cN_\l \dotplus \hat\cN_{\ov\l}), \qquad \cT_a\cap
(\hat\cN_\l \dotplus \hat\cN_{\ov\l})=\{0\}\oplus \cN = \{\{0,f\}:
f\in\cN\},\quad \l\in\CR,
\end{gather}
and the following equality holds
\begin{gather}\label{4.58}
\dim (\tma/\cT_a)=\dim (\dom\tma/\dom\cT_a)= N_+ +N_- -k_\cN.
\end{gather}
\begin{proof}
Let $\cN'\subset \cN$ be the subspace \eqref{4.45.2}. Then by \eqref{4.55}
$\hat\cN'\subset \hat\cN_\l \dotplus \hat\cN_{\ov\l}$ and the first equality in
\eqref{4.51} together with \eqref{4.45.4} gives the first equality in
\eqref{4.57}.

Next assume that $\{y,f\}\in \cT_a\cap (\hat\cN_\l \dotplus \hat\cN_{\ov\l})$.
Since $\cT_a\subset\tmi$, it follows from \eqref{4.51} that $y\in\cN'$ and
$f\in\cN$. Moreover, by \eqref{4.28} $y(a)=0$ and therefore $y=0$. Hence
$\{y,f\}\in \{0\}\oplus \cN$. Conversely, in view of the second equality in
\eqref{4.51} each pair $\{0,f\}$ with $f\in\cN$ belongs to $\hat \cN_\l
\dotplus \hat\cN_{\ov\l}$ and obviously  $\{0,f\}\in\cT_a$. Hence
$\{0,f\}\in\cT_a\cap (\hat \cN_\l \dotplus \hat\cN_{\ov\l})$, which yields the
second equality in \eqref{4.57}. Finally, one proves formula \eqref{4.58} in
the same way as \eqref{4.52}.
\end{proof}
\end{proposition}
\begin{proposition}\label{pr4.16}
Assume that the canonical system \eqref{4.1} has the regular endpoint $a$.
Moreover, let $\cT_1$ be the linear relation \eqref{4.24} and let
$\bH_1=\{y(a):\, y\in\dom\cT_1\}$.  Then $\tmi\subset \cT_1\subset\tma$ and
\begin{gather} \label{4.60}
\bH_1=(J\bH_a)^\perp,\qquad \dim (\dom \cT_1 / \dom\cT_a)=n-k_\cN,
\end{gather}
where  $\bH_a$ and $k_\cN$ are defined by \eqref{4.1.4} and \eqref{4.1.5}
respectively.
\end{proposition}
\begin{proof}
It follows from \eqref{4.45.1} that $(Jy(a), z(a))=0$ for any $y\in\dom\cT_1$
and $z\in\cN$. Therefore $\bH_1\subset (J\bH_a)^\perp$ and to prove the first
equality in \eqref{4.60} it remains to show that $(J\bH_a)^\perp\subset \bH_1$.

First assume that the system \eqref{4.1} is regular and let
\begin{gather}\label{4.61a}
\bH_1'=\{y(a):\, y\in\dom\tma \;\;\text{and}\;\; y(b)=0\}.
\end{gather}
Moreover, let $\cN_0'$ be the subspace in $\cN_0(=\ker \tma)$ given by
$\cN_0'=\{y\in\cN_0:\, y(a)\in \bH_a^\perp\}$. Then $\cN_0'\cap \cN=\{0\}$ and,
therefore, the equality $(y,y)_\D=0\; (y\in\cN_0')$ implies $y=0$. Hence
$\cN_0'$ is a finite-dimensional Hilbert space with the inner product
$(y,z)_\D$.

Let $h\in (J\bH_a)^\perp$, so that $Jh\in \bH_a^\perp$. Then
$\f(z)=-(Jh,z(a)),\;\; z\in\cN_0' $ is an  antilinear functional on $\cN_0'$
and hence there exists $f_h\in\cN_0'$ such that
\begin{gather} \label{4.62}
(f_h,z)_\D=-(J h,z(a)), \quad z\in\cN_0'.
\end{gather}
Next assume that $y\in\AC$ is the solution of the equation $Jy'-B(t)y=\D(t)
f_h(t)$ such that $y(b)=0$. Then $\{y,f_h\}\in\tma$ and, consequently, $y(a)\in
\bH_1' $. Therefore $y(a)\in (J\bH_a)^\perp$, which gives the inclusion $J
y(a)\in \bH_a^\perp$. Applying now the Lagrange's  identity \eqref{4.10} to
$\{y,f_h\} $ and $\{z,0\}\; (z\in\cN_0')$ and taking \eqref{4.62} into account
one obtains
\begin{gather*}
-(J h,z(a))=(f_h,z)_\D=-(J y(a), z(a)), \quad z\in\cN_0'.
\end{gather*}
In this equality $Jh\in\bH_a^\perp, \; J y(a)\in\bH_a^\perp$ and $z(a)$ takes
on any values from $\bH_a^\perp$, when $z$ run through $\cN_0'$. Therefore
$y(a)=h$ and, consequently, $h\in\bH_1'$. Thus $(J\bH_a)^\perp\subset \bH_1'$,
which together with the obvious inclusion $\bH_1'\subset\bH_1(\subset
(J\bH_a)^\perp)$ gives
\begin{gather} \label{4.63}
\bH_1=\bH_1'=(J\bH_a)^\perp.
\end{gather}

Assume now that the system \eqref{4.1} is singular. For each finite segment
$\cI'=[a,\b]\subset\cI$ denote by $\cN^{\cI'}$ the linear space \eqref{4.19}
and let $\bH_a^{\cI'}=\{y(a): y\in\cN^{\cI'}\}$. It is easily seen that
$\bH_a=\bigcap_{\cI'\subset\cI}\bH_a^{\cI'}$. Moreover, it was shown in the
proof of Lemma \ref{lem4.4} that there is a segment $\cI_0'=[a,\b_0]$ such that
$\bH_a^{\cI_0'}\subset \bH_a^{\cI'}$ for all $\cI'\subset \cI_0'$ and
$\bH_a^{\cI_0'} = \bH_a^{\cI'}$ for all $\cI'\supset \cI_0'$. This implies that
$\bH_a^{\cI_0'}= \bH_a$.

Next assume that $\tma^{\cI_0'}$ is the maximal relation in $\cL_\D^2(\cI_0')$
induced by the restriction of the system \eqref{4.1} onto $\cI_0'$. Since this
restriction is regular, it follows from \eqref{4.63} and \eqref{4.61a} that for
any $h\in (J\bH_a)^\perp (= (J\bH_a^{\cI_0'})^\perp)$ there exists $\{y,f\}\in
\tma^{\cI_0'}$ such that $y(0)=h$ and $y(\b_0)=0$. Continuing the functions $y$
and $f$ by $0$ onto $\cI$ we obtain the pair $\{y,f\}\in\cT_1$ with $y(0)=h$.
This yields the required inclusion $(J\bH_a)^\perp\subset \bH_1$.

Let us  prove the second  equality in \eqref{4.60}. It follows from the first
equality in \eqref{4.60} that $r_1:=\dim\bH_1=n-k_\cN$. Let $\{y_j\}_1^{r_1}$
be a system of functions $y_j\in \dom \cT_1$ such that $\{y_j(0)\}_1^{r_1}$ is
a basis in $\bH_1$. Then the immediate checking shows that this system forms a
basis of $\dom\cT_1$ modulo $\dom\cT_a$, which yields the desired equality.
\end{proof}
\begin{definition}
The canonical system \eqref{4.1} is called  definite if the corresponding null
manifold $\cN=\{0\}$.
\end{definition}
The following corollaries are implied by the above results on arbitrary (not
necessarily definite) canonical systems.
\begin{corollary}\label{cor4.18}
If the system \eqref{4.1} is definite, then $N_\pm=n_\pm$ and the following
Neumann formula holds
\begin{gather} \label{4.64}
\tma=\tmi\dotplus \hat \cN_\l\dotplus\hat\cN_{\ov\l}, \qquad \l\in\CR.
\end{gather}
\begin{proof}
The desired statements are immediate from Propositions \ref{pr4.13} and
\ref{pr4.14}.
\end{proof}
\end{corollary}
\begin{corollary}\label{cor4.19}
Let the canonical system \eqref{4.1} with the regular endpoint $a$ be definite.
Then $\cT_a=\tmi$ and for every $h\in\bH$ there exists $\{y,f\}\in\tma$ such
that $y(a)=h$. If in addition the system is regular (that is, $\cI=[a,b]$),
then for any $h_1, \; h_2\in\bH $ there is $\{y,f\}\in\tma$ such that
$y(a)=h_1$ and $y(b)=h_2$.
\end{corollary}
\begin{proof}
The first statement follows from Propositions \ref{pr4.11a} and \ref{pr4.16}.

Next assume that the system is definite and regular and let $h_1,h_2\in\bH$.
Then by \eqref{4.61a} and \eqref{4.63} there is $\{y_1,f_1\}\in \tma$ with
$y_1(a)=h_1$ and $y_1(b)=0$. Moreover, by symmetry there is $\{y_2,f_2\}\in
\tma$ with $y_2(a)=0$ and $y_2(b)=h_2$. Clearly, the sum $\{y,f\}=\{y_1,f_1\}
+\{y_2,f_2\}$ has the required properties.
\end{proof}
\begin{remark}\label{rem4.20}
For the definite system \eqref{4.1} Theorem \ref{th4.11} and Corollary
\ref{cor4.19} were proved in \cite{Orc}; the Neumann formula \eqref{4.64} was
obtained in \cite{LesMal03}.

The  general (not necessarily definite) canonical system of an arbitrary order
$n$ was considered in \cite{LesMal03}, where the minimal relation $\Tmi$ was
defined as closure of $T_0$ (see \eqref{4.8}) and then the equality
$\Tmi^*(=T_0^*)=\Tma$ was proved . Note in this connection that our definition
\eqref{4.6}, \eqref{4.7} of $\Tmi$ seems to be more natural  and convenient for
applications; in particular cases of differential operators and definite
canonical systems such a representation of $\Tmi$ can be found, e.g., in
\cite{Nai,BHSW10}. Observe also that  our Proposition \ref{pr4.16} improves
similar result in \cite[Proposition 2.12]{LesMal03}.
\end{remark}
\section{Boundary relations for canonical systems  and boundary conditions}
\subsection{Boundary bilinear forms}
In this section we suppose that the  canonical systems \eqref{4.1} is defined
on the interval $\cI=[a, b\rangle $ with the regular endpoint $a$.

As is known the signature operator  in \eqref{4.1} is unitary equivalent to
\begin{gather} \label{5.1}
J=\begin{pmatrix} 0 & 0&-I_H \cr 0& i\d I_{\hat H}&0\cr I_H&
0&0\end{pmatrix}:H\oplus\hat H\oplus H \to H\oplus\hat H\oplus H,
\end{gather}
where $\d\in\{-1,1\}$ and $H,\;\hat H$ are finite-dimensional  Hilbert spaces.
The numbers $\d, \dim H$ and $\dim\hat H$ are unitary invariants of $J$, which
are defined by the following relations: if we let
\begin{gather}\label{5.3}
\nu_+=\dim\ker (i J-I) \;\;\; \text{and} \;\;\; \nu_-=\dim\ker (i J+I),
\end{gather}
then
\begin{gather}\label{5.3a}
\d=\sign (\nu_- -\nu_+), \quad \dim H=\text{min} \{\nu_+,\nu_-\}, \quad  \dim
\hat H=|\nu_- -\nu_+|.
\end{gather}
Using this fact we assume without loss of generality that
\begin{gather}\label{5.4}
\bH=H\oplus\hat H \oplus H
\end{gather}
and the signature operator $J$ in \eqref{4.1} is given by \eqref{5.1}.

Next consider the boundary bilinear form $[\cd,\cd]_b$ on $\dom \tma $ defined
by \eqref{4.4}. Clearly, this form is skew-Hermitian (that is
$[y,z]_b=-\ov{[z.y]}_b$) and its kernel coincides with $\dom\cT_1$, where
$\cT_1$ is the linear relation \eqref{4.24}. Moreover, since  $\dom \tmi\subset
\dom\cT_1\subset\dom \tma $ and by \eqref{4.52} $\dim (\dom\tma /
\dom\tmi)<\infty$, there exists a (not unique) direct decomposition
\begin{gather}\label{5.7}
\dom\tma=\dom \cT_1\dotplus \cD_{b+}\dotplus \cD_{b-}
\end{gather}
such that
\begin{gather}\label{5.9}
\nu_{b+}:=\dim \cD_{b+}<\infty, \quad \nu_{b-}:=\dim \cD_{b-}<\infty
\end{gather}
and the following relations  are valid
\begin{gather}\label{5.7a}
\im [y,y]_b>0, \;\; 0\neq y\in\cD_{b+}; \quad \im [z,z]_b<0, \;\; 0\neq
z\in\cD_{b-}; \quad [y,z]_b=0, \;\; y\in\cD_{b+}, \; z\in\cD_{b-}.
\end{gather}
As is known the numbers $\nu_{b+} $ and $\nu_{b-}$ are called indices if
inertia of the form $[\cd,\cd]_b$.These numbers are uniquely defined by the
form and do not depend on the choice of the decomposition \eqref{5.7}.
\begin{lemma}\label{lem5.1}
Let $[\cd,\cd]_b$ be the bilinear form \eqref{4.4} with the indices of inertia
\eqref{5.9} and let $\d_b:=\sign (\nu_{b+}-\nu_{b-})$. Then: 1) there exists
Hilbert spaces $\cH_b$ and $\hat\cH_b$ and a surjective linear map
\begin{gather}\label{5.13}
\G_b=(\G_{0b}:\,  \hat\G_b:\,  \G_{1b})^\top:\dom\tma\to
\cH_b\oplus\hat\cH_b\oplus \cH_b
\end{gather}
such that
\begin{gather}\label{5.15}
[y,z]_b=i\d_b (\hat\G_b y, \hat\G_b z)-(\G_{1b}y,\G_{0b}z)+(\G_{0b}y,\G_{1b}z),
\quad y,z \in \dom\tma.
\end{gather}
Letting $\bH_b:=\cH_b\oplus\hat\cH_b\oplus \cH_b$ and introducing the signature
operator $J_b\in [\bH_b]$ by
\begin{gather}\label{5.16a}
J_b=\begin{pmatrix} 0 & 0&-I_{\cH_b} \cr 0& i\d_b I_{\hat \cH_b}&0\cr
I_{\cH_b}& 0&0\end{pmatrix}:\cH_b\oplus\hat \cH_b\oplus \cH_b \to
\cH_b\oplus\hat \cH_b\oplus \cH_b
\end{gather}
one can represent the identity \eqref{5.15} as
\begin{gather}\label{5.16b}
[y,z]_b=(J_b \G_b y, \G_b z)_{\bH_b},\quad y,z \in \dom\tma.
\end{gather}

2) if a surjective linear map $\G_b$ of the form \eqref{5.13} satisfies
\eqref{5.15} , then $\ker\G_b=\dom\cT_1$ and
\begin{gather}\label{5.17}
\dim\cH_b=\text{min} \{\nu_{b+}, \nu_{b-}\}, \qquad \dim \hat\cH_b=|\nu_{b+}-
\nu_{b-}|.
\end{gather}
\begin{proof}
1) Assume for definiteness that $\nu_{b+}\geq \nu_{b-}$, so that $\d_b=1$. It
follows from \eqref{5.7a} that $\cD_{b+}$  and $\cD_{b-}$ are
finite-dimensional  Hilbert spaces with the inner products $(y_1,y_2)_+=-i[y_1,
y_2]_b, \; y_1, y_2\in \cD_{b+}$ and $(z_1,z_2)_-=i[z_1,z_2]_b, \; z_1,z_2\in
\cD_{b-}$ respectively. Moreover, by \eqref{5.7}
\begin{gather}\label{5.18}
\dom\tma=\dom \cT_1\dotplus (\hat\cH_b\oplus\cH_b )\dotplus \cD_{b-},
\end{gather}
where $\hat\cH_b$ and $\cH_b$ are subspaces in $\cD_{b+}$ such that
$\dim\cH_b=\nu_{b-}(=\dim \cD_{b-})$ and $\cD_{b+}=\hat\cH_b\oplus\cH_b$.

Let $V$ be a unitary operator from $\cD_{b-}$ onto $\cH_b$ and let
\begin{gather}\label{5.19}
\hat\G_b= \cP_{\hat\cH_b}, \qquad \G_{0b}=\tfrac1 {\sqrt 2}(\cP_{\cH_b}+V
\cP_{\cD_{b-}}),  \qquad  \G_{1b}=-\tfrac {i} {\sqrt 2}(\cP_{\cH_b}-V
\cP_{\cD_{b-}}),
\end{gather}
where $\cP_{\hat\cH_b}, \;\cP_{\cH_b}$ and $\cP_{\cD_{b-}}$ are the skew
projections onto the subspaces $\hat\cH_b, \; \cH_b$ and $\cD_{b-}$
corresponding to the decomposition \eqref{5.18}. The immediate checking shows
that the map $\G_b$ given by \eqref{5.13} and \eqref{5.19} is surjective and
satisfies \eqref{5.15}.

Similarly one proves the statement 1)  in the case $\nu_{b+}< \nu_{b-}$.

The statement 2) immediately follows from surjectivity of $\G_b$ and the
identity \eqref{5.15}.
\end{proof}
\end{lemma}
\begin{remark}\label{rem5.2}
One can show that the map $\G_b$ in Lemma \ref{lem5.1} can be represented in
the more explicit form. Namely, it is not difficult to prove that there exist
systems of functions $\{\psi_j\}_1^{\nu_b},\;\{\f_j\}_1^{\hat\nu_b} $ and
$\{\t_j\}_1^{\nu_b}$ in $\dom\tma$ with $\nu_b=\text{min} \{\nu_{b+},
\nu_{b-}\}$ and $\hat\nu_b=|\nu_{b+}- \nu_{b-}|$ such that  the
operators
\begin{gather*}
\G_{0b}y =\{[y,\psi_j ]_b\}_1^{\nu_b}, \quad   \hat\G_b y=\{[y,\f_j
]_b\}_1^{\hat\nu_b}, \quad \G_{1b}y =\{[y,\t_j ]_b\}_1^{\nu_b}, \quad
y\in\dom\tma
\end{gather*}
form the surjective linear map $\G_b=(\G_{0b}:\,  \hat\G_b:\,
\G_{1b})^\top:\dom\tma\to \bC^{\nu_b}\oplus\bC^{\hat \nu_b}\oplus \bC^{\nu_b}$
satisfying the identity \eqref{5.15}. This assertion  shows that for each
$y\in\dom\tma$ the elements $\G_{0b} y, \G_{1b}y$ and $\hat\G_b y$ are, in
fact, boundary values of the function $y(\cd)$ at the endpoint $b$.
\end{remark}
\subsection{Decomposing boundary relations}
 Assume without loss of generality that the Hilbert space $\bH$ and the signature
operator $J$ in \eqref{4.1} are defined by \eqref{5.4} and \eqref{5.1}
respectively. In this case each function $y(\cd)\in\dom\tma$ admits the
representation
\begin{gather}\label{5.21}
y(t)=\{y_0(t),\, \hat y(t),\, y_1(t)\}(\in \bH), \quad t\in\cI,
\end{gather}
where $y_0(t), \hat y(t)$ and $y_1(t)$ are components of $y(t)$ corresponding
to the decomposition \eqref{5.4}.

Let $\nu_+$ and $\nu_-$ by given by \eqref{5.3} and let $\nu_{b+}$ and
$\nu_{b-}$ be indices of inertia \eqref{5.9}. Then according to Lemma
\ref{lem5.1} there exist Hilbert spaces $\cH_b$ and $\hat\cH_b$ satisfying
\eqref{5.17} and the surjective linear map $\G_b=(\G_{0b}:\,  \hat\G_b:\,
\G_{1b})^\top$ such that \eqref{5.15} holds. Without loss of generality assume
that
\begin{gather}\label{5.22}
\nu_{b+}-\nu_{b-}\geq \nu_- -\nu_+
\end{gather}
and consider the following three alternative cases:

(i) $\nu_--\nu_+\geq 0$

It follows from \eqref{5.3a} and \eqref{5.17} that in this case
\begin{gather}\label{5.23}
\dim H=\nu_+, \quad \dim\hat H=\nu_- -\nu_+, \quad \dim\cH_b=\nu_{b-}, \quad
\dim\hat\cH_b=\nu_{b+}-\nu_{b-}
\end{gather}
and the inequality \eqref{5.22} gives $\dim \hat H \leq \dim\hat\cH_b$.
Therefore without loss of generality we can assume that $\hat H$ is a subspace
in $\hat\cH_b$. Letting $\cH_2=\hat\cH_b\ominus \hat H$, we obtain
$\hat\cH_b=\cH_2\oplus \hat H$, so that the operator  $\hat \G_b$ in
\eqref{5.13} admits the block representation $\hat \G_b=(\G_{2b}:\,
\hat\G_{1b})^\top: \dom\tma\to\cH_2\oplus \hat H$. Put
\begin{gather*}
\cH_1=H\oplus \hat H\oplus \cH_b,\quad \cH_0=\cH_2\oplus\cH_1= \cH_2\oplus
H\oplus \hat H\oplus \cH_b
\end{gather*}
and introduce the operators
\begin{gather}
\G_0'y=\{ \G_{2b}y,\, y_0(a),\, \tfrac i {\sqrt 2}(\hat
y(a)-\hat\G_{1b}y)\},\,\G_{0b}y\} (\in\cH_2\oplus H\oplus  \hat
H\oplus \cH_b),\label{5.25}\\
\G_1'y=\{y_1(a),\,  \tfrac 1 {\sqrt 2}(\hat y(a)+\hat\G_{1b}y),\, -\G_{1b}y\}
(\in H\oplus \hat H \oplus \cH_b),\quad y\in\dom\tma \label{5.26}.
\end{gather}

(ii) $\nu_- -\nu_+ < 0$ and $\nu_{b+}-\nu_{b-}\geq 0 $, so that
\begin{gather}\label{5.27}
\dim H=\nu_-, \quad \dim\hat H=\nu_+ -\nu_-, \quad \dim\cH_b=\nu_{b-}, \quad
\dim\hat\cH_b=\nu_{b+}-\nu_{b-}.
\end{gather}
In this case we put
\begin{gather}
 \cH_2=\hat H\oplus \hat\cH_b, \quad \cH_1=H \oplus \cH_b,\quad
\cH_0= \cH_2\oplus \cH_1=(\hat H\oplus \hat\cH_b)\oplus H \oplus \cH_b\nonumber\\
\G_0'y=\{\{\hat y(a),\, \hat\G_b y\},\, y_0(a),\, \G_{0b}y\} (\in (\hat
H\oplus\hat\cH_b)\oplus H\oplus  \cH_b),\label{5.28}\\
\G_1'y=\{y_1(a),  -\G_{1b}y\} (\in H \oplus \cH_b),\quad
y\in\dom\tma\label{5.29}.
\end{gather}

 (iii) $\nu_{b+}-\nu_{b-}< 0 $, so that
\begin{gather}\label{5.30}
\dim H=\nu_-, \quad \dim\hat H=\nu_+ -\nu_-, \quad \dim\cH_b=\nu_{b+}, \quad
\dim\hat\cH_b=\nu_{b-}-\nu_{b+}.
\end{gather}
 In view of \eqref{5.22} one has $\dim \hat \cH_b \leq \dim \hat H$, which enables us
to assume by analogy with the case (i) that  $\hat\cH_b\subset \hat H $.
Letting $\cH_2=\hat H\ominus \cH_b$ one obtains  $\hat H=\cH_2\oplus
\hat\cH_b$, which implies the representation $\hat y(t)=\{\hat y_2(t),\, \hat
y_b(t)\}(\in \cH_2\oplus \hat\cH_b)$ of the functions $\hat y(t)$ from
\eqref{5.21}.

In the case (iii) we let
\begin{gather}
\cH_1=H\oplus \hat\cH_b \oplus \cH_b,\quad \cH_0=\cH_2\oplus\cH_1=\cH_2\oplus
H\oplus \hat\cH_b \oplus \cH_b\nonumber\\
\G_0'y=\{\hat y_2(a),\, y_0(a),\,-\tfrac i {\sqrt 2}(\hat y_b(a)-\hat\G_b
y)\},\,\G_{0b}y\}(\in \cH_2\oplus H\oplus \hat \cH_b \oplus \cH_b)\label{5.31}\\
 \G_1'y=\{y_1(a),  \tfrac 1 {\sqrt
2}(\hat y_b(a)+\hat\G_b y),-\G_{1b}y\} (\in H\oplus \hat\cH_b \oplus
\cH_b),\quad y\in\dom\tma\label{5.32}.
\end{gather}

Note that in each of the cases (i)--(iii) $\cH_1$ is a subspace in $\cH_0$,
$\cH_2=\cH_0\ominus \cH_1$ and $\G_j'$ is a linear map from $\dom\tma$ to
$\cH_j, \; j\in \{0.1\}$. Moreover, formulas \eqref{5.23}, \eqref{5.27} and
\eqref{5.30} imply that in all cases (i)--(iii)

\begin{gather}\label{5.32a}
\dim\cH_0=\nu_+ +\nu_{b+}, \qquad \dim\cH_1=\nu_- +\nu_{b-}.
\end{gather}
\begin{theorem}\label{th5.3}
Assume that $a$ is a regular endpoint for the canonical system \eqref{4.1}  and
the inequality \eqref{5.22} is satisfied. Moreover, let $\cH_j$ be Hilbert
spaces and $\G_j':\dom\tma\to \cH_j, \; j\in\{0,1\}$ be linear maps constructed
for the alternative cases (i)--(iii) just before the theorem. Then the equality
\begin{gather}\label{5.33}
\G=\lb \lb {\pi y\choose \pi f}, {\G_0' y\choose \G_1'y } \rb
:\{y,f\}\in\tma\rb
\end{gather}
defines the boundary relation $\G:(\LI)^2\to \HH$ for $\Tma(=\Tmi^*)$ with
\begin{gather}\label{5.34}
\dim\cH_0=N_+ \quad\emph{and} \quad \dim\cH_1=N_-.
\end{gather}
\end{theorem}
\begin{proof}
Let us show that the linear relation \eqref{5.33} satisfies the assumptions of
Corollary \ref{cor3.6} for $A:=\Tmi$.

Assume that $\G'=(\G_0':\, \G_1')^\top: \dom\tma\to \HH$. Then definitions
\eqref{5.25}--\eqref{5.32} of $\G_0'$ and $\G_1'$ and the equality $\ker
\G_b=\dom\cT_1$ (see Lemma \ref{lem5.1}, 2)) imply that $\ker \G'=\cT_a$.
Therefore by \eqref{5.33} one has $\ker\G=\wt\pi\,\ker \G'=\wt\pi \cT_a=\Tmi$.
Moreover, it follows from \eqref{5.33} that $\dom\G=\wt\pi\tma =\Tma $.

Next, the immediate calculations with taking \eqref{5.15} into account show
that in each of the cases (i)--(iii) the operators $\G_0'$ and $\G_1'$ satisfy
the relation
\begin{gather*}
[y,z]_b- (J y(a), z(a))=(\G_1'y, \G_0'z)-(\G_0'y, \G_1'z)+ i (P_2\G_0'y,
P_2\G_0'z), \quad  y,z\in\dom\tma.
\end{gather*}
This  and the Lagrange's identity \eqref{4.5} give the identity \eqref{3.7} for
$\G$.

Now it remains to prove  \eqref{3.23}. It follows from \eqref{5.7} and
\eqref{5.9} that
\begin{gather*}
\nu_{b+}+\nu_{b-}=\dim (\dom\tma / \dom \cT_1)=\dim (\dom\tma / \dom
\cT_a)-\dim (\dom\cT_1 / \dom \cT_a).
\end{gather*}
Combining this equality with \eqref{4.58} and the second equality in
\eqref{4.60} one obtains
\begin{gather*}
\nu_{b+}+\nu_{b-}=(N_+ +N_- -k_\cN) -(n-k_\cN)=N_++N_--n.
\end{gather*}
This and \eqref{5.32a} give
\begin{gather}\label{5.35}
\dim (\HH)=(\nu_++\nu_-)+(\nu_{b+}+\nu_{b-})=n+(N_++N_--n)=N_++N_-.
\end{gather}

Next, in view of \eqref{5.33} one has $\mul\G=\{\G'y:\, \{y,f\}\in \ker
(\wt\pi\up \tma)\;\;\text{for some}\;\;f\in\lI\}$ and \eqref{4.45.6} yields
\begin{gather}\label{5.35.1}
\mul\G=\G'\cN=\{\{\G_0'y,\G_1' y\}:\, y\in\cN\}.
\end{gather}
Since obviously $\ker (\G'\up\cN)=\{0\}$, it follows from \eqref{5.35.1} that
\begin{gather}\label{5.36}
n_\G(=\dim (\mul\G))=\dim\cN=k_N.
\end{gather}
This and \eqref{4.49} imply that
\begin{gather}\label{5.37}
n_+ +n_- + 2n_\G=(N_+ -k_\cN)+(N_- -k_\cN)+2k_\cN=N_+ +N_-.
\end{gather}
Combining now  \eqref{5.35} and \eqref{5.37} we arrive at the required equality
\begin{gather*}
\dim (\HH)=n_+ +n_- + 2n_\G.
\end{gather*}
Thus according to Corollary \ref{cor3.6} formula \eqref{5.33} defines the
boundary relation $\G$ for $\Tma$. Moreover, combining \eqref{3.22} with
\eqref{5.36} and \eqref{4.49} we obtain the equalities \eqref{5.34}.
\end{proof}
\begin{definition}\label{def5.4a}
The boundary relation $\G:(\LI)^2\to \HH$ constructed in Theorem \ref{th5.3}
will be called a decomposing  boundary relation  for $\Tma$.
\end{definition}
\begin{proposition}\label{pr5.4}
The formal deficiency indices of the canonical system \eqref{4.1} with the
regular endpoint $a$ can be calculated via
\begin{gather}\label{5.38}
N_+=\nu_+ +\nu_{b+}, \qquad N_-=\nu_- +\nu_{b-},
\end{gather}
where $\nu_\pm$ are the numbers \eqref{5.3} and $\nu_{b+}, \; \nu_{b-}$ are
indices of inertia of the bilinear form $[\cd,\cd]_b$. It follows from
\eqref{5.38} that in the case of the regular endpoint $a$ the following
inequalities hold
\begin{gather}\label{5.39}
\nu_+\leq N_+\leq n, \qquad \nu_-\leq N_-\leq n.
\end{gather}
\end{proposition}
\begin{proof}
If $\nu_\pm$ and $\nu_{b\pm}$ satisfy \eqref{5.22}, then the equalities
\eqref{5.38} follow from \eqref{5.34} and \eqref{5.32a}. In the opposite case
$\nu_{b+}-\nu_{b-}< \nu_- -\nu_+$ the equalities \eqref{5.38} can be obtained
by passage to the system
\begin{gather*}
-J y'(t)+B(t)y(t)=\D(t)f(t).
\end{gather*}
\end{proof}
In the case $N_+=N_-$ the construction of the decomposing boundary relation for
$\Tma$ can be rather simplified. Namely, the following corollary is valid.
\begin{corollary}\label{cor5.5}
Assume that $a$ is a regular endpoint for the canonical system \eqref{4.1}.
Then:

1) this system has equal deficiency indices $N_+=N_-$  if and only if
\begin{gather}\label{5.40}
\nu_{b+}-\nu_{b-}=\nu_- -\nu_+
\end{gather}
2) if $N_+=N_-$, then there exist a Hilbert space $\cH_b$ with
$\dim\cH_b=\text{min}\{\nu_{b+},\nu_{b-}\}$ and a surjective linear map
\begin{gather}\label{5.41}
\G_b=(\G_{0b}:\,  \hat\G_b:\,  \G_{1b})^\top:\dom\tma\to \cH_b\oplus\hat
H\oplus \cH_b
\end{gather}
 such that the identity \eqref{5.15} holds with $\d_b=\d(=\sign (\nu_- -
 \nu_+))$. Moreover, for each such a map $\G_b$ the equality
\begin{gather}\label{5.42}
\G=\lb \lb {\pi y\choose \pi f}, {\{ y_0(a),\, \tfrac i {\sqrt 2}\d (\hat
y(a)-\hat\G_b y),\,\G_{0b}y \}\choose \{ y_1(a),\, \tfrac 1 {\sqrt 2} (\hat
y(a)+\hat\G_b y),\,-\G_{1b}y \} } \rb :\{y,f\}\in\tma\rb
\end{gather}
defines the decomposing boundary relation  $\G:(\LI)^2 \to (H\oplus\hat H\oplus
\cH_b )^2$ for $\Tma$.

In the case of the regular system \eqref{4.1} one can put $\cH_b= H$ and
\begin{gather}\label{5.43}
\G=\lb \lb {\pi y\choose \pi f}, {\{ y_0(a),\, \tfrac i {\sqrt 2}\d (\hat
y(a)-\hat y(b)),\, y_0(b) \}\choose \{ y_1(a),\, \tfrac 1 {\sqrt 2} (\hat
y(a)+\hat y(b)),\,-y_1(b) \} } \rb :\{y,f\}\in\tma\rb.
\end{gather}
\end{corollary}
\begin{proof}
The statement 1) follows from \eqref{5.38}.

2) Combining \eqref{5.40} with \eqref{5.3a} and \eqref{5.17} one obtains $\dim
\hat\cH_b=\dim \hat  H$. Therefore one can put in \eqref{5.13} $\hat\cH_b= \hat
H$, in which case the map $\G_b$ takes the form \eqref{5.41} and the equality
\eqref{5.33} for $\G$ can be represented as \eqref{5.42}. In the case of the
regular system one can put $\cH_b=H$ and $\G_b y=\{y_0(b),\, \hat y(b), \,
y_1(b)\}(\in\bH)$, so that the equality \eqref{5.42} takes the form
\eqref{5.43}.
\end{proof}
\begin{corollary}\label{cor5.6}
Assume that the canonical system \eqref{4.1} with the regular endpoint $a$ has
minimal formal deficiency indices $N_+=\nu_+$ and $N_-=\nu_-$. If $N_+ \geq
N_-$, then the equality
\begin{gather}\label{5.44}
\G=\lb \lb {\pi y\choose \pi f}, {\{ \hat y(a), \,y_0(a)\}\choose  y_1(a) } \rb
:\{y,f\}\in\tma\rb
\end{gather}
defines the decomposing boundary relation $\G:(\LI)^2\to (\hat H\oplus H)\oplus
H $ for $\Tma$.
\end{corollary}
\begin{proof} Since $\nu_{b+}=\nu_{b-}=0$, it follows from \eqref{5.17} that
$\cH_b=\hat\cH_b=\{0\}$. Combining this equalities with \eqref{5.28},
\eqref{5.29} and \eqref{5.33} we obtain the representation \eqref{5.44} for
$\G$.
\end{proof}
\subsection{Boundary conditions for definite systems}
As is known (see for instance \cite{LesMal03})the maximal operator $\Tma$
induced by the definite system \eqref{4.1} possesses the following property:
for each $\{\wt y, \wt f\}\in \Tma$ there exists a unique function $y\in\AC$
such that $y\in\wt y$ and $\{y,f\}\in\tma$ for each $f\in\wt f$. Bellow,
without any additional comments, we associate such a function $y\in\AC$ with
each pair $\{\wt y, \wt f\}\in \Tma$.
\begin{theorem}\label{th5.9}
Let under the conditions of Theorem \ref{th5.3} the canonical system
\eqref{4.1} be definite. Then:

1) The operators $\G_j:\Tma\to\cH_j, \; j\in \{0,1\}$ given by
\begin{gather}\label{5.52}
\G_0\{\wt y, \wt f\}=\G_0'y, \qquad \G_1\{\wt y, \wt f\}=\G_1'y, \qquad \{\wt
y, \wt f\}\in\Tma
\end{gather}
form the boundary triplet $\Pi=\bta$ for $\Tma$.

In the case of minimal deficiency indices $n_+=\nu_+$ and $n_-=\nu_-$ one has
$\cH_0=\hat H\oplus H, \; \cH_1=H$ and the equality \eqref{5.52} takes the form
\begin{gather}\label{5.52a}
\G_0\{\wt y, \wt f\}= \{\hat y(a), \, y_0(a)\}(\in\hat H\oplus H), \quad
\G_1\{\wt y, \wt f\}= y_1(a)(\in H), \quad \{\wt y, \wt f\} \in \Tma.
\end{gather}

2) If $n_+=n_-$, then the statement 2) of Corollary \ref{cor5.5} holds and the
decomposing boundary relation \eqref{5.42} turns into the boundary  triplet
$\Pi=\{\cH,\G_0, \G_1\}$ for $\Tma$ with $\cH=H\oplus\hat H \oplus \cH_b$ and
the operators $\G_j: \Tma\to \cH$ given by
\begin{gather}
\G_0 \{\wt y, \wt f\} =\{ y_0(a),\, \tfrac i {\sqrt 2}\d (\hat y(a)-\hat\G_b
y),\,\G_{0b}y \} (\in H\oplus\hat H \oplus \cH_b),\qquad\qquad\qquad\quad\label{5.53}\\
\G_1 \{\wt y, \wt f\} = \{ y_1(a),\, \tfrac 1 {\sqrt 2} (\hat y(a)+\hat\G_b
y),\,-\G_{1b}y \}(\in H\oplus\hat H \oplus \cH_b), \quad \{\wt y, \wt f\} \in
\Tma, \label{5.54}
\end{gather}
with $\d=\sign (\nu_--\nu_+)$. In the case of the regular system \eqref{4.1}
one can put $\cH= H\oplus\hat H \oplus H$ and
\begin{gather}
\G_0 \{\wt y, \wt f\}=\{ y_0(a),\, \tfrac i {\sqrt 2}\d (\hat y(a)-\hat
y(b)),\, y_0(b) \}(\in H\oplus\hat H \oplus H),\qquad\qquad\qquad \label{5.55} \\
\G_1 \{\wt y, \wt f\}= \{ y_1(a),\, \tfrac 1 {\sqrt 2}(\hat y(a)+\hat
y(b)),\,-y_1(b) \}(\in H\oplus\hat H \oplus H), \quad \{\wt y, \wt f\} \in
\Tma.\label{5.56}.
\end{gather}
\end{theorem}
\begin{proof}
1) Let $\G$ be the decomposing boundary relation \eqref{5.33} for $\Tma$. Then
by \eqref{5.35.1} $\mul \G=\{0\}$ and Corollary \ref{cor3.5}, 2) implies that
the operators \eqref{5.52} form the boundary triplet $\Pi=\bta$ for $\Tma$.
Moreover, in the case $n_\pm=\nu_\pm$ the equality \eqref{5.44} gives
\eqref{5.52a}.

The statement 2) of the theorem follows from Corollary \ref{cor5.5}, 2).
\end{proof}
In the sequel the boundary triplet $\Pi=\bta$ defined in Theorem \ref{th5.9},
1) will be called the decomposing boundary triplet for $\Tma$. In the case of
equal deficiency indices $n_+=n_-$ such a triplet takes the form
$\Pi=\{\cH,\G_0, \G_1\}$, where $\cH=H\oplus\hat H \oplus \cH_b$ and $\G_0, \;
\G_1$ are defined by \eqref{5.53} and \eqref{5.54}.
\begin{proposition}\label{pr5.10}
Let the minimal relation $\Tmi$ induced by the definite  system \eqref{4.1}
with the regular endpoint $a$ has equal deficiency indices $n_+=n_-$ and let
$\Pi=\bt$ be the decomposing boundary triplet \eqref{5.53}, \eqref{5.54} for
$\Tma$. Then for each operator pair (linear relation) $\t=\{(C_0,C_1);\cK\}$
given by the block representations
\begin{gather}\label{5.57}
C_0=(C_{0a}:\hat C_a:C_{0b}):H\oplus\hat H \oplus \cH_b\to\cK, \qquad
C_1=(C_{1a}:\hat C_b:C_{1b}):H\oplus\hat H \oplus \cH_b\to\cK
\end{gather}
the equality (the boundary conditions)
\begin{gather}\label{5.58}
\wt A=\{\{\wt y, \wt f\}\in\Tma:\,C_{0a}y_0(a)+\hat C_a \hat y(a)+
C_{1a}y_1(a)+C_{0b}\G_{0b}y+\hat C_b \hat \G_b y + C_{1b}\G_{1b}y=0 \}
\end{gather}
defines a proper extension $\wt A$ of $\Tmi$ and, conversely, for each such an
extension there is a unique  admissible  operator pair (linear relation)
$\t=\{(C_0,C_1);\cK\}$ given by \eqref{5.57} and such that \eqref{5.58} holds.
Moreover, the extension \eqref{5.58} is maximal dissipative, maximal
accumulative or self-adjoint if and only if the operator pair (linear relation)
$\wt\t=\{(\wt C_0,\wt C_1);\cK\}$ with
\begin{gather}\label{5.59}
\wt C_0=(C_{0a}: -\tfrac {i\d}{\sqrt 2} (\hat C_a-\hat C_b):C_{0b}), \qquad \wt
C_1=(C_{1a}: \tfrac 1 {\sqrt 2} (\hat C_a+\hat C_b): -C_{1b})
\end{gather}
is maximal dissipative, maximal accumulative or self-adjoint respectively.
\end{proposition}
\begin{proof}
It follows from \eqref{5.53} and \eqref{5.59} that the boundary conditions
\eqref{5.58} can be written as
\begin{gather}\label{5.59a}
\wt A=\{\{\wt y, \wt f\}\in\Tma:\, \wt C_0 \G_0 \{\wt y, \wt f\}+\wt C_1 \G_1
\{\wt y, \wt f\}=0 \}.
\end{gather}
This and Proposition \ref{pr2.5} yield the desired statements.
\end{proof}
In the following corollary we give a somewhat different description of proper
extensions $\wt A\in Ext_{\Tma} $.
\begin{corollary}\label{cor5.10a}
Assume that $a$ is a regular endpoint for the definite canonical system
\eqref{4.1} and $n_+=n_-=:m$. Let $\bH_b$ be a Hilbert space, let $J_b\in
[\bH_b] $ be a signature operator and let $\G_b:\dom\tma \to \bH_b$ be a
surjective linear map such that \eqref{5.16b} holds (according to Lemma
\ref{lem5.1} such $\bH_b,\; J_b$ and $\G_b$ exist and $\dim \bH_b=\nu_{b+}+
\nu_{b+}$). Moreover, let $\cK$ be a Hilbert space with $\dim\cK=m$, let
$C_a\in [\bH,\cK]$ and $C_b\in [\bH_b,\cK] $ be operators such that $\ran
(C_a:\,C_b)=\cK$ and let $\wt A\in Ext _{\Tmi}$ be an extension given by
\begin{gather}\label{5.59b}
\wt A=\{\{\wt y, \wt f\}\in\Tma: \, C_a y(a)+C_b \G_b y=0\}.
\end{gather}
Then $\wt A$ is maximal dissipative, maximal accumulative or self-adjoint if
and only if
\begin{gather}\label{5.59c}
i(C_a J C_a^*-C_b J_b C_b^*)\leq 0, \quad i(C_a J C_a^*-C_b J_b C_b^*)\geq 0
\;\;\;\text{or} \;\;\; C_a J C_a^*=C_b J_b C_b^*
\end{gather}
respectively.
\end{corollary}
\begin{proof}
Assume without loss of generality that $\bH_b=\cH_b\oplus \hat H \oplus\cH_b $
and the operators $J_b$ and $\G_b$ are of the form \eqref{5.16a} and
\eqref{5.13} respectively (with $\hat H$ in place of $\hat\cH_b$). Then
according to Theorem \ref{th5.9} the Hilbert space $\cH=H\oplus \hat H\oplus
\cH_b$ and the operators \eqref{5.53}, \eqref{5.54} form the (decomposing)
boundary triplet $\Pi=\bt$ for $\Tma$. Next assume that
\begin{gather*}
C_a=(C_{0a}: \hat C_a :C_{1a}):H\oplus\hat H \oplus H\to\cK, \quad
C_b=(C_{0b}:\hat C_b: C_{1b}):\cH_b\oplus\hat H \oplus \cH_b\to\cK
\end{gather*}
are the block representations of $C_a$ and $C_b$ and let $\wt C_0$ and $\wt
C_1$ be given by \eqref{5.59}. Then \eqref{5.59b} can be  written as
\eqref{5.58} and according to Proposition \ref{pr5.10} $\wt A$ is maximal
dissipative, maximal accumulative or self-adjoint if and only if the operator
pair $\t==\{(\wt C_0,\wt C_1);\cK\}$ belongs to the same class. The immediate
calculations show that
\begin{gather*}
2\im (\wt C_1 \wt C_0^*)=i(C_b J_b C_b^*-C_a J C_a^*).
\end{gather*}
Moreover, since $\ran (C_a:\,C_b)=\cK$, it follows that the operator pair $(\wt
C_0:\,\wt C_1 )$ is admissible. Applying now Proposition \ref{pr2.2}, 2) we
arrive at the required statement.
 \end{proof}
\begin{definition}\label{def5.11}
The boundary conditions \eqref{5.58} are said to be separated  if there exists
a decomposition $\cK=\cK_a\oplus\cK_b$ such that the operators \eqref{5.57} are
\begin{gather}
C_0=\begin{pmatrix} N_{0a} & \hat N_a & 0 \cr 0 & 0 & N_{0b}
\end{pmatrix}:H\oplus\hat H \oplus \cH_b\to\cK_a\oplus\cK_b\label{5.60}\\
C_1=\begin{pmatrix} N_{1a} & 0 & 0 \cr 0 & \hat N_b & N_{1b}
\end{pmatrix}:H\oplus\hat H \oplus \cH_b\to\cK_a\oplus\cK_b\label{5.61}
\end{gather}
and, consequently, the equality \eqref{5.58} takes the form
\begin{gather}\label{5.62}
\wt A=\{\{\wt y, \wt f\}\in\Tma:\,N_{0a}y_0(a)+\hat N_a \hat y(a)+
N_{1a}y_1(a)=0, \;\; N_{0b}\G_{0b}y+\hat N_b \hat \G_b y + N_{1b}\G_{1b}y=0 \}.
\end{gather}

The separated boundary conditions \eqref{5.62} will by called maximal
dissipative, maximal accumulative or self-adjoint if they define the extension
$\wt A $ of the corresponding class.
\end{definition}
With the separated boundary conditions \eqref{5.62} we associate the operators
\begin{gather}
S_a=\im (N_{1a}N_{0a}^*)+\tfrac 1 2 \hat N_a \hat N_a^*, \qquad  S_b=\im
(N_{1b}N_{0b}^*)+\tfrac 1 2 \hat N_b \hat N_b^*,\label{5.63}\\
\wt N_a=(N_{0a}-i N_{1a}\, :\, -i\sqrt{2} \hat N_a):H\oplus\hat H\to \cK_a,
\quad \wt N_b=(i\sqrt{2} \hat N_b\, :\,N_{0b}-i N_{1b} ): \hat H\oplus
\cH_b\to\cK_b.\label{5.64}
\end{gather}
\begin{theorem}\label{th5.12}
Let for simplicity $\nu_-\geq \nu_+$ and let the assumptions of Proposition
\ref{pr5.10} be satisfied. Then: 1) the separated boundary conditions defined
by \eqref{5.60}--\eqref{5.62} are maximal dissipative if and only if
\begin{gather}\label{5.65}
S_a\geq 0, \quad S_b\leq 0 \;\;\; \emph {and} \;\;\; 0\in \rho (\wt N_a)\cap
\rho (N_{0b}+i N_{1b}),
\end{gather}
in which case the following equalities hold
\begin{gather}\label{5.65a}
\dim\cK_a=\nu_-, \qquad \dim\cK_b=\nu_{b-}.
\end{gather}
The same boundary conditions are maximal accumulative if and only if
\begin{gather}\label{5.66}
S_a\leq 0, \quad S_b\geq 0 \;\;\; \emph {and} \;\;\; 0\in \rho (N_{0a}+i
N_{1a})\cap\rho(\wt N_b),
\end{gather}
in which  case
\begin{gather}\label{5.66a}
\dim\cK_a=\nu_+, \qquad \dim\cK_b=\nu_{b+}
\end{gather}
(here $\cK_a$ and $\cK_b$ are Hilbert spaces from \eqref{5.60} and
\eqref{5.61}).

2) self-adjoint separated boundary conditions  exist if and only if
$\nu_-=\nu_+$ or, equivalently, if and only if $\bH=H\oplus H $ and the
operator $J$ in \eqref{4.1} is
\begin{gather*}
J=\begin{pmatrix} 0 & -I_H \cr I_H & 0 \end{pmatrix}:H\oplus H\to H\oplus H.
\end{gather*}
If this condition is satisfied, then:

(i) the decomposing boundary triplet \eqref{5.53}, \eqref{5.54} takes the form
$\Pi=\bt$, where $\cH=H\oplus\cH_b$ and the operators $\G_j, \; j\in\{0,1\}$
are given by
\begin{gather}\label{5.67}
\G_0 \{\wt y,\wt f\}=\{y_0(a), \G_{0b}y\}(\in H\oplus\cH_b), \quad \G_1 \{\wt
y,\wt f\}=\{y_1(a), -\G_{1b}y\}(\in H\oplus\cH_b), \quad\{\wt y,\wt f\}\in\Tma
\end{gather}

(ii) the general form of  self-adjoint separated  boundary conditions is
\begin{gather}\label{5.68}
\wt A=\{\{\wt y, \wt f\}\in\Tma:\,N_{0a}y_0(a)+ N_{1a}y_1(a)=0, \;\;
N_{0b}\G_{0b}y+ N_{1b}\G_{1b}y=0 \},
\end{gather}
where the operators $N_{ja}\in [H,\cK_a]$ and $N_{jb}\in [\cH_b,\cK_b], \;
j\in\{0,1\}$ are components of  self-adjoint operator pairs $\t_a=\{(N_{0a},
N_{1a} ); \cK_a\}$ and $\t_b=\{(N_{0b}, N_{1b} ); \cK_b\}$.
\end{theorem}
\begin{proof}
1) Let $\wt C_0$ and $\wt C_1$ be the operators \eqref{5.59} corresponding to
the separated boundary conditions \eqref{5.62}. Then in view of \eqref{5.60}
and \eqref{5.61} one has
\begin{gather}
\wt C_0=\begin{pmatrix} N_{0a} & -\frac i {\sqrt 2} \hat N_a & 0 \cr 0 &\frac i
{\sqrt 2} \hat N_b  & N_{0b}
\end{pmatrix}:H\oplus\hat H \oplus \cH_b\to\cK_a\oplus\cK_b\label{5.69}\\
\wt C_1=\begin{pmatrix} N_{1a} & \frac 1 {\sqrt 2} \hat N_a & 0 \cr 0 & \frac 1
{\sqrt 2}\hat N_b & -N_{1b}
\end{pmatrix}:H\oplus\hat H \oplus \cH_b\to\cK_a\oplus\cK_b\label{5.70}
\end{gather}
Combining now the last statement in Proposition \ref{pr5.10} with formulas
\eqref{2.2.1} and \eqref{2.2.2} we obtain the following assertion:

(a) the boundary conditions \eqref{5.62} are maximal dissipative (resp. maximal
accumulative) if and only if $\im (\wt C_1 \wt C_0^*)\geq 0$ and $0\in\rho (\wt
C_0-i\wt C_1 )$ (resp. $\im (\wt C_1 \wt C_0^*)\leq 0$ and $0\in\rho (\wt
C_0+i\wt C_1 )$ ).

It follows from \eqref{5.69} and \eqref{5.70} that
\begin{gather*}
\wt C_1 \wt C_0^*=\begin{pmatrix} N_{1a}N_{0a}^*+\tfrac i 2 \hat N_a \hat N_a^*
& -\tfrac i 2 \hat N_a \hat N_b^* \cr \tfrac i 2 \hat N_b \hat N_a^* &
-N_{1b}N_{0b}^*-\tfrac i 2 \hat N_b \hat N_b^*
\end{pmatrix}
\end{gather*}
and, consequently, $ \im (\wt C_1 \wt C_0^*)=\text{diag}(S_a, -S_b)$. Hence the
following equivalences are valid
\begin{gather}\label{5.71}
\im (\wt C_1 \wt C_0^*)\geq 0 \iff S_a\geq 0 \;\;\text{and}\;\; S_b\leq 0;
\quad \im (\wt C_1 \wt C_0^*)\leq 0 \iff S_a\leq 0 \;\;\text{and}\;\; S_b\geq
0.
\end{gather}
Moreover, by \eqref{5.69} and \eqref{5.70} one has
\begin{gather*}
\wt C_0-i\wt C_1 =\begin{pmatrix}\wt N_a & 0 \cr 0 & N_{0b}+i
N_{1b}\end{pmatrix}:(H\oplus\hat H)\oplus \cH_b\to \cK_a\oplus\cK_b,\\
\wt C_0+i\wt C_1 =\begin{pmatrix} N_{0a}+i N_{1a} & 0 \cr 0 & \wt N_b
\end{pmatrix}:H\oplus(\hat H\oplus \cH_b)\to
\cK_a\oplus\cK_b,
\end{gather*}
which yields the equivalences
\begin{gather}\label{5.72}
0\in\rho (\wt C_0-i\wt C_1 )\Leftrightarrow 0\in \rho (\wt N_a)\cap \rho
(N_{0b}+i N_{1b}),\quad 0\in\rho (\wt C_0+i\wt C_1 )\Leftrightarrow 0\in\rho
(N_{0a}+i N_{1a}) \cap \rho (\wt N_b).
\end{gather}
Now  assertion (a) together with \eqref{5.71} and \eqref{5.72} gives the
required description of all maximal dissipative and  accumulative separated
boundary conditions by means of \eqref{5.65} and \eqref{5.66}. Moreover,
\eqref{5.65} implies that
\begin{gather*}
\dim\cK_a=\dim (H\oplus\hat H), \qquad \dim\cK_b=\dim\cH_b,
\end{gather*}
which in view of \eqref{5.3a} and \eqref{5.17} leads to \eqref{5.65a}.
Similarly one proves the equalities \eqref{5.66a}.

2) Since self-adjoint separated boundary conditions are simultaneously maximal
dissipative and maximal accumulative, it follows from \eqref{5.65a} and
\eqref{5.66a} that an existence of such conditions yields the equality
$\nu_-=\nu_+$. Moreover, if this equality is satisfied, then the general form
\eqref{5.68} of self-adjoint separated boundary conditions follows from the
statement 1) of the theorem.
\end{proof}
\begin{remark}\label{rem5.13}
1) Theorem \ref{th5.12} enables one to introduce the important class of maximal
accumulative (dissipative) separated boundary conditions, which consist of the
self-adjoint condition at the regular endpoint $a$ and  the maximal
accumulative (dissipative) condition at the point $b$. If, for instance,
$\nu_-\geq \nu_+$, then such separated conditions are defined by
\begin{gather*}
\wt A=\{\{\wt y, \wt f\}\in\Tma:\,N_{0a}y_0(a)+ N_{1a}y_1(a)=0, \;\;
N_{0b}\G_{0b}y+\hat N_b \hat \G_b y + N_{1b}\G_{1b}y=0 \},
\end{gather*}
where the operators $N_{0a}$ and  $N_{0b}$ form the self-adjoint pair
$\t_a=\{(N_{0a}, N_{1a} ); \cK_a\} $, while the operators $N_{0b}, \; N_{1b}$
and $\hat N_b$ form the maximal accumulative pair $\t_b=\{((\tfrac i {\sqrt
2}\hat N_b :N_{0b}),( \tfrac 1 {\sqrt 2}\hat N_b:-N_{1b} )); \cK_b\}$.

2) For a regular definite system \eqref{4.1} one can put in Corollary
\ref{cor5.10a} $\bH_b=\bH, \; J_b=J$ and $\G_b y=y(b), y\in\dom\tma$, in which
case this corollary gives the following well known statement \cite{GK,Orc}: the
extension $\wt A= \{\{\wt y, \wt f\}\in\Tma :\, C_a y(a)+C_b y(b)=0\}$ is
self-adjoint if and only if $C_aJC_a^*=C_bJC_b^*$. The case of the singular
endpoint $b$ under the additional assumptions $\nu_+=\nu_-$ and $\mul
\Tma=\{0\}$ was considered in the paper \cite{Kra89}, where the criterium for
self-adjointness  of the boundary condition \eqref{5.59b} in the form of the
last equality in \eqref{5.59c} was obtained. Note in this connection that our
approach based on the concept of a decomposing boundary triplet seems to be
more convenient. In particular, such an approach made it possible to describe
in Theorem \ref{th5.12} various classes of separated boundary conditions.
 \end{remark}

\end{document}